\documentclass[reqno,a4paper]{amsart}
\usepackage{amssymb}
\usepackage{amsfonts}
\usepackage{geometry}

\newtheorem{theorem}{Theorem}
\theoremstyle{plain}

\newtheorem{proposition}{Proposition}

\numberwithin{equation}{section}
\input{tcilatex}

\begin{document}
\title[Must Derive the Uehling-Uhlenbeck Equation below $W^{4,1}$]{On the
Weak Coupling Limit of Quantum Many-body Dynamics and the Quantum Boltzmann
Equation}
\author{Xuwen Chen}
\address{Department of Mathematics, Brown University, 151 Thayer Street,
Providence, RI 02912}
\email{chenxuwen@math.brown.edu}
\urladdr{http://www.math.brown.edu/\symbol{126}chenxuwen/}
\author{Yan Guo}
\address{Division of Applied Mathematics, Brown University, 182 George
Street, Providence, RI 02912}
\email{yan\_guo@brown.edu}
\urladdr{http://www.cfm.brown.edu/people/guoy/}

\begin{abstract}
The rigorous derivation of the Uehling-Uhlenbeck equation from more
fundamental quantum many-particle systems is a challenging open problem in
mathematics. In this paper, we exam the weak coupling limit of quantum $N$%
-particle dynamics. We assume the integral of the microscopic interaction is
zero and we assume $W^{4,1}$ per-particle regularity on the coressponding
BBGKY sequence so that we can rigorously commute limits and integrals. We
prove that, if the BBGKY sequence does converge in some weak sense, then
this weak-coupling limit must satisfy the infinite quantum Maxwell-Boltzmann
hierarchy instead of the expected infinite Uehling-Uhlenbeck hierarchy,
regardless of the statistics the particles obey. Our result indicates that,
in order to derive the Uehling-Uhlenbeck equation, one must work with
per-particle regularity bound below $W^{4,1}$.
\end{abstract}

\maketitle

\section{Introduction}

The rigorous derivation of the celebrated Uehling-Uhlenbeck equation from
more fundamental quantum many-particle systems is a challenging open problem
in mathematics. This problem has received a lot of attentions in recent
years. In particular, Erd\"{o}s, Salmhofer and Yau have given, in \cite%
{ErdosFock}, a formal derivation of the spatially homogeneous
Uehling-Uhlenbeck equation as the thermodynamic limit from the Fock space
model. Around the same time, in \cite{BCEP1,BCEP2,BCEP4}, Benedetto,
Castella, Esposito, and Pulvirenti initiated a different study of the
problem with the "classical $N$-particle
Bogoliubov--Born--Green--Kirkwood--Yvon (BBGKY) hierarchy" approach. Here,
"classical BBGKY hierarchy" means the usual BBGKY hierarchy in $\mathbb{R}%
^{3N+1}$. Moreover, Benedetto, Castella, Esposito, and Pulvirenti consider
the $N\rightarrow \infty $ limit of $N$ particles in $\mathbb{R}^{3}$
instead of the thermodynamic limit. In this paper, we follow the classical
BBGKY hierarchy approach in \cite{BCEP1,BCEP2,BCEP4}. Let $t\in \mathbb{R}$, 
$\mathbf{x}_{k}=\left( x_{1},...,x_{k}\right) ,\mathbf{v}_{k}=\left(
v_{1},...,v_{k}\right) \in \mathbb{R}^{3k}$, $\varepsilon =N^{-\frac{1}{3}}$%
, and $\phi $ be an even pair interaction. We consider the following quantum
BBGKY hierarchy%
\begin{equation}
\left( \partial _{t}+\mathbf{v}_{k}\cdot \nabla _{\mathbf{x}_{k}}\right)
f_{N}^{(k)}=\frac{1}{\sqrt{\varepsilon }}A_{\varepsilon }^{(k)}f_{N}^{(k)}+%
\frac{N}{\sqrt{\varepsilon }}B_{\varepsilon }^{(k+1)}f_{N}^{(k+1)},
\label{hierarchy:QB-BBGKY in Differential Form}
\end{equation}%
where%
\begin{eqnarray*}
\frac{1}{\sqrt{\varepsilon }}A_{\varepsilon }^{(k)} &=&\sum_{1\leqslant
i<j\leqslant k}\frac{1}{\sqrt{\varepsilon }}A_{i,j}^{\varepsilon }, \\
\frac{N}{\sqrt{\varepsilon }}B_{\varepsilon }^{(k+1)} &=&\sum_{j=1}^{k}\frac{%
N}{\sqrt{\varepsilon }}B_{j,k+1}^{\varepsilon },
\end{eqnarray*}%
with%
\begin{eqnarray*}
\frac{1}{\sqrt{\varepsilon }}A_{i,j}^{\varepsilon }f_{N}^{(k)} &=&\frac{-i}{%
\sqrt{\varepsilon }}\frac{1}{\left( 2\pi \right) ^{3}}\sum_{\sigma =\pm
1}\sigma \int_{\mathbb{R}^{3}}e^{\frac{ih\cdot (x_{i}-x_{j})}{\varepsilon }}%
\hat{\phi}(h) \\
&&\times f_{N}^{(k)}\left( t,\mathbf{x}_{k},v_{1}...,v_{i-1},v_{i}-\sigma 
\frac{h}{2},v_{i+1},...,v_{j-1},v_{j}+\sigma \frac{h}{2},v_{j+1},...,v_{k}%
\right) dh, \\
\frac{N}{\sqrt{\varepsilon }}B_{j,k+1}^{\varepsilon }f_{N}^{(k+1)} &=&-i%
\frac{N}{\sqrt{\varepsilon }}\frac{1}{\left( 2\pi \right) ^{3}}\sum_{\sigma
=\pm 1}\sigma \int_{\mathbb{R}^{3}}dx_{k+1}\int_{\mathbb{R}%
^{3}}dv_{k+1}\int_{\mathbb{R}^{3}}e^{\frac{ih\cdot (x_{j}-x_{k+1})}{%
\varepsilon }}\hat{\phi}(h) \\
&&\times f_{N}^{(k+1)}\left( t,\mathbf{x}%
_{k},x_{k+1},v_{1}...,v_{j-1},v_{j}-\sigma \frac{h}{2},v_{j+1},...,v_{k+1}+%
\sigma \frac{h}{2}\right) dh.
\end{eqnarray*}%
We would like to immediately remark that, we have not assumed anything about
the statistics the particles obey, though it seems that writing down
hierarchy (\ref{hierarchy:QB-BBGKY in Differential Form}) like \cite[(2.18)]%
{BCEP2} suggests that we are assuming either the Bose-Einstein or the
Fermi-Dirac statistics. One can check from \cite[(2.18-2.22)]{BCEP2} and 
\cite[(2.10-2.16)]{BCEP4} that the BBGKY hierarchy for the Bose-Einstein /
Fermi-Dirac statistics and for the Maxwell-Boltzmann statistics are
identical, though their initial data are totally different. To be precise,
the coefficient of $B_{\varepsilon }^{(k+1)}f_{N}^{(k+1)}$ is $\frac{%
\varepsilon ^{-3}}{\sqrt{\varepsilon }}$ in \cite[(2.18)]{BCEP2} while the
coefficient of $B_{\varepsilon }^{(k+1)}f_{N}^{(k+1)}$ in \cite[(2.14)]%
{BCEP4} is $\frac{N-k}{\sqrt{\varepsilon }}.$ Since $N=\varepsilon ^{-3}$,
the difference $\frac{-k}{\sqrt{\varepsilon }}B_{j,k+1}^{\varepsilon }$ must
tend to zero as long as $\frac{N}{\sqrt{\varepsilon }}B_{j,k+1}^{\varepsilon
}=\frac{\varepsilon ^{-3}}{\sqrt{\varepsilon }}B_{j,k+1}^{\varepsilon }$
tends to a definite limit as $\varepsilon \rightarrow 0$ for every fixed $k$%
. Hence, we have not assumed anything about the statistics the particles
obey.

We shall not go into the details about the rise of hierarchy (\ref%
{hierarchy:QB-BBGKY in Differential Form}). We refer the interested readers
to \cite{BCEP2} and \cite{BCEP4}. The $\varepsilon \rightarrow 0$ limit of
hierarchy (\ref{hierarchy:QB-BBGKY in Differential Form}) is called the
weak-coupling limit of quantum many-body dynamics. As mentioned in \cite%
{BCEP4}, this is characterized by the fact that the potential interaction is
weak in the sense that it is of order $\sqrt{\varepsilon }$ and the density
of particles is $1$. Therefore the number of collisions per unit time is $%
\varepsilon ^{-1}$. Since the quantum mechanical cross--section in the Born
approximation (justified because the potential is small) is quadratic in the
potential interaction, the accumulated effect is of the order 
\begin{equation*}
\text{number of collisions}\ \times \lbrack \text{potential interaction}%
]^{2}=1/\varepsilon \times \varepsilon =1.
\end{equation*}

We are concerned with the central question of identifying the weak coupling
limit $\varepsilon \rightarrow $ $0$ (or equivalently $N\rightarrow \infty $%
) for such a quantum BBGKY hierarchy (\ref{hierarchy:QB-BBGKY in
Differential Form}), even at a formal level. The expected $\varepsilon
\rightarrow $ $0$ limit of hierarchy (\ref{hierarchy:QB-BBGKY in
Differential Form}) is the infinite Uehling-Uhlenbeck hierarchy which is
defined by%
\begin{equation}
\left( \partial _{t}+\mathbf{v}_{k}\cdot \nabla _{\mathbf{x}_{k}}\right)
f^{(k)}=\sum_{j=1}^{k}Q_{1,j,k+1}f^{(k+1)}+\sum_{j=1}^{k}Q_{2,j,k+2}f^{(k+2)}
\label{hierarchy:U-U}
\end{equation}%
where the two particle term $Q_{1,j,k+1}$ is given by 
\begin{eqnarray*}
Q_{1,j,k+1}f^{(k+1)}\left( \mathbf{x}_{k},\mathbf{v}_{k}\right)  &=&\int
dv_{j}^{\prime }dv_{k+1}dv_{k+1}^{\prime }W\left(
v_{j},v_{k+1}|v_{j}^{\prime },v_{k+1}^{\prime }\right)  \\
&&\times \{f^{(k+1)}(\mathbf{x}_{k},x_{j},v_{1},...,v_{j}^{\prime
},...,v_{k+1}^{\prime }) \\
&&-f^{(k+1)}(\mathbf{x}_{k},x_{j},v_{1},...,v_{k+1})\},
\end{eqnarray*}%
and the three particle term $Q_{2,j,k+2}$ is given by 
\begin{eqnarray*}
Q_{2,j,k+2}f^{(k+2)}\left( \mathbf{x}_{k},\mathbf{v}_{k}\right)  &=&8\pi
^{3}\theta \int dv_{j}^{\prime }dv_{k+1}dv_{k+1}^{\prime }W\left(
v_{j},v_{k+1}|v_{j}^{\prime },v_{k+1}^{\prime }\right)  \\
&&\times \{f^{(k+2)}(\mathbf{x}_{k},x_{j},x_{j},v_{1},...,v_{j}^{\prime
},...,v_{k+1}^{\prime },v_{j}) \\
&&+f^{(k+2)}(\mathbf{x}_{k},x_{j},x_{j},v_{1},...,v_{j}^{\prime
},...,v_{k+1}^{\prime },v_{k+1}) \\
&&-f^{(k+2)}(\mathbf{x}_{k},x_{j},x_{j},v_{1},...,v_{k+1},v_{j}^{\prime }) \\
&&-f^{(k+2)}(\mathbf{x}_{k},x_{j},x_{j},v_{1},...,v_{k+1},v_{k+1}^{\prime
})\}.
\end{eqnarray*}%
In the above,%
\begin{eqnarray*}
W\left( v,v^{\prime }|v_{\ast },v_{\ast }^{\prime }\right)  &=&\frac{1}{8\pi
^{2}}\left[ \hat{\phi}(v^{\prime }-v)+\theta \hat{\phi}(v^{\prime }-v_{\ast
})\right] ^{2}\delta (v+v_{\ast }-v^{\prime }-v_{\ast }^{\prime }) \\
&&\times \delta (\frac{1}{2}\left( v^{2}+v_{\ast }^{2}-\left( v^{\prime
}\right) ^{2}-\left( v_{\ast }^{\prime }\right) ^{2}\right) ).
\end{eqnarray*}%
and $\theta =\pm 1$ for bosons and fermions respectively. The expected
mean-field equation for the infinite Uehling-Uhlenbeck hierarchy (\ref%
{hierarchy:U-U}) is exactly the Uehling-Uhlenbeck equation \cite{U-U}. It
arises as the special solution%
\begin{equation*}
f^{(k)}(t,\mathbf{x}_{k},\mathbf{v}_{k})=\dprod%
\limits_{j=1}^{k}f(t,x_{j},v_{j})
\end{equation*}%
to hierarchy (\ref{hierarchy:U-U}), provided that $f$ satisfies

\begin{eqnarray}
&&\partial _{t}f+v\cdot \nabla _{x}f  \label{eqn:U-U} \\
&=&\int dv_{\ast }dv_{\ast }^{\prime }dv^{\prime }W\left( v,v_{\ast
}|v^{\prime },v_{\ast }^{\prime }\right)  \notag \\
&&\left\{ f^{\prime }f_{\ast }^{\prime }\left( 1+8\pi ^{3}\theta f\right)
\left( 1+8\pi ^{3}\theta f_{\ast }\right) -ff_{\ast }\left( 1+8\pi
^{3}\theta f^{\prime }\right) \left( 1+8\pi ^{3}\theta f_{\ast }^{\prime
}\right) \right\} .  \notag
\end{eqnarray}

However, to our surprise, we find that, as long as $\left\{
f_{N}^{(k+1)}\right\} $ is of $W^{4,1}$ per-particle regularity, the
infinite Uehling-Uhlenbeck hierarchy (\ref{hierarchy:U-U}) is not the $%
\varepsilon \rightarrow 0$ limit of the BBGKY hierarchy (\ref%
{hierarchy:QB-BBGKY in Differential Form}) regardless of the statistics the
particles obey. We find that, regardless of the statistics the particles
obey, the $\varepsilon \rightarrow 0$ limit of the BBGKY hierarchy (\ref%
{hierarchy:QB-BBGKY in Differential Form}) is the infinite quantum
Maxwell-Boltzmann hierarchy coming from \cite{BCEP4}, defined as (\ref%
{hierarchy:QBHierarchy in differential form}) in this paper. In fact, if one
formally commutes integrals and $\lim_{\varepsilon \rightarrow 0}$ in
approriate places, one finds that, regardless of the statistics the
particles obey, the formal $\varepsilon \rightarrow 0$ limit of the BBGKY
hierarchy (\ref{hierarchy:QB-BBGKY in Differential Form}) must be the
infinite quantum Maxwell-Boltzmann hierarchy (\ref{hierarchy:QBHierarchy in
differential form}) instead of the infinite Uehling-Uhlenbeck hierarchy (\ref%
{hierarchy:U-U}).

To this end, we define the quantum Maxwell-Boltzmann hierarchy to be%
\begin{equation}
\left( \partial _{t}+\mathbf{v}_{k}\cdot \nabla _{\mathbf{x}_{k}}\right)
f^{(k)}=\sum_{j=1}^{k}C_{j,k+1}f^{(k+1)}
\label{hierarchy:QBHierarchy in differential form}
\end{equation}%
with%
\begin{eqnarray}
C_{j,k+1}f^{(k+1)} &=&\frac{1}{8\pi ^{2}}\int_{\mathbb{R}^{3}}dv_{j+1}\int_{%
\mathbb{S}^{2}}dS_{\omega }\left\vert \omega \cdot \left(
v_{j}-v_{k+1}\right) \right\vert \left\vert \hat{\phi}\left( \left( \omega
\cdot \left( v_{j}-v_{k+1}\right) \right) \omega \right) \right\vert ^{2}
\label{def:collision operator} \\
&&\times \lbrack f^{(k+1)}\left( t,\mathbf{x}%
_{k},x_{j},v_{1}...,v_{j-1},v_{j}^{\prime },v_{j+1},...,v_{k+1}^{\prime
}\right) -  \notag \\
&&f^{(k+1)}\left( t,\mathbf{x}%
_{k},x_{j},v_{1}...,v_{j-1},v_{j},v_{j+1},...,v_{k+1}\right) ].  \notag
\end{eqnarray}%
The $C_{j,k+1}$ in (\ref{hierarchy:QBHierarchy in differential form}) is
certainly the Boltzmann collision operator. We use $A^{\varepsilon }$ and $%
B^{\varepsilon }$ to denote the inhomogeneous terms in (\ref%
{hierarchy:QB-BBGKY in Differential Form}) because neither of the $%
\varepsilon \rightarrow 0$ limit of $A^{\varepsilon }$ nor $B^{\varepsilon }$
along gives the collision operator $C.$ The collision operator $C_{j,k+1}$
in (\ref{hierarchy:QBHierarchy in differential form}) arises as the $%
\varepsilon \rightarrow 0$ limit of a suitable composition of $%
A^{\varepsilon }$ and $B^{\varepsilon }.$

Notice that, the mean-field equation of hierarchy (\ref%
{hierarchy:QBHierarchy in differential form}) is not the Uehling-Uhlenbeck
equation (\ref{eqn:U-U}). If one assumes Maxwell-Boltzmann statistics as
well as $\hat{\phi}(0)=0,$ a term by term convergence from (\ref%
{hierarchy:QB-BBGKY in Differential Form}) to (\ref{hierarchy:QBHierarchy in
differential form}) was rigorously established in \cite{BCEP4}. The
mean-field equation in this case is the quantum Boltzmann equation:%
\begin{equation}
\left( \partial _{t}+v\cdot \nabla _{x}\right) f=Q(f,f)
\label{eqn:proposed QBEquation}
\end{equation}%
where the collision operator $Q$ is given by%
\begin{eqnarray*}
Q(f,f) &=&\frac{1}{8\pi ^{2}}\int_{\mathbb{R}^{3}}dv_{1}\int_{\mathbb{S}%
^{2}}dS_{\omega }\left\vert \omega \cdot \left( v-v_{1}\right) \right\vert
\left\vert \hat{\phi}\left( \left( \omega \cdot \left( v-v_{1}\right)
\right) \omega \right) \right\vert ^{2} \\
&&\left[ f\left( t,x,v^{\prime }\right) f(t,x,v_{1}^{\prime
})-f(t,x,v)f(t,x,v_{1})\right] ,
\end{eqnarray*}%
and $\hat{\phi}$ is the Fourier transform of $\phi $, and%
\begin{equation*}
v^{\prime }=v-\left( \left[ v-v_{1}\right] \cdot \omega \right) \omega ,%
\text{ \ \ \ \ }v_{1}^{\prime }=v_{1}+\left( \left[ v-v_{1}\right] \cdot
\omega \right) \omega ,
\end{equation*}%
because%
\begin{equation*}
f^{(k)}(t,\mathbf{x}_{k},\mathbf{v}_{k})=\dprod%
\limits_{j=1}^{k}f(t,x_{j},v_{j})
\end{equation*}%
is a solution to hierarchy (\ref{hierarchy:QBHierarchy in differential form}%
) provided that $f$ solves (\ref{eqn:proposed QBEquation}).

In our main theorem, we assume $W^{4,1}$ per-particle regularity on the
BBGKY sequence $\left\{ f_{N}^{(k)}\right\} _{k=1}^{N}$ so that we can
rigorously commute limits and integrals in suitable places. We are then able
to prove that, if the BBGKY sequence $\left\{ f_{N}^{(k)}\right\} _{k=1}^{N}$
does converge in some weak sense, then the limit sequence $\left\{
f^{(k)}=\lim_{N\rightarrow \infty }f_{N}^{(k)}\right\} _{k=1}^{\infty }$
must satisfy the infinite quantum Maxwell-Boltzmann hierarchy (\ref%
{hierarchy:QBHierarchy in differential form}) instead of the infinite
Uehling-Uhlenbeck hierarchy (\ref{hierarchy:U-U}), regardless of the
statistics the particles obey. We work in the space $W_{k}^{4,1}$ which is $%
W^{4,1}(\mathbb{R}^{3k}\times \mathbb{R}^{3k})$ equiped with the weak
topology. We work with the norm%
\begin{equation*}
\left\Vert f^{(k)}\right\Vert
_{W_{k}^{4,1}}=\sum_{j=1}^{k}\sum_{m=0}^{4}\left( \left\Vert \partial
_{x_{j}}^{m}f^{(k)}\right\Vert _{L^{1}}+\left\Vert \partial
_{v_{j}}^{m}f^{(k)}\right\Vert _{L^{1}}\right) .
\end{equation*}%
Our main theorem is the following.

\begin{theorem}[Main Theorem]
\label{THM:MainTheorem}Assume the interaction potential $\phi $ is an even
Schwarz class function and satisfies the vanishing condition: $\hat{\phi}$
vanishes at the origin to at least 11th order. Suppose a subsequence of $%
\left\{ \Gamma _{N}=\left\{ f_{N}^{(k)}\right\} _{k=1}^{N}\right\} _{N}$
converges weakly to some $\Gamma =\left\{ f^{(k)}\right\} _{k=1}^{\infty }$
in the following sense:

(1) In $C\left( \left[ 0,T\right] ,W_{k}^{4,1}\right) $, we have,%
\begin{equation*}
f_{N}^{(k)}\rightarrow f^{(k)}\text{ as }N\rightarrow \infty ,
\end{equation*}

(2) There is a $C>0$ such that%
\begin{equation*}
\sup_{k,N,t\in \left[ 0,T\right] }\frac{1}{k}\left\Vert
f_{N}^{(k)}\right\Vert _{W_{k}^{4,1}}\leqslant C.
\end{equation*}
Then $\Gamma =\left\{ f^{(k)}\right\} _{k=1}^{\infty }$ satisfies the
infinite quantum Boltzmann hierarchy (\ref{hierarchy:QBHierarchy in
differential form}), regardless of the form of the initial datum $\left\{
f_{N}^{(k)}\left( 0\right) \right\} _{k=1}^{N}$ $\ $or the statistics
(Bose-Einstein / Fermi-Dirac / Maxwell-Boltzmann) it satisfies. In
particular, $f^{(k)}$ does not satisfy the infinite Uehling-Uhlenbeck
hierarchy (\ref{hierarchy:U-U}).
\end{theorem}

We remark that Theorem \ref{THM:MainTheorem} certainly does not imply that
the Uehling-Uhlenbeck equation (\ref{eqn:U-U}) is not derivable as a
mean-field limit. Our result is merely an indication that, in order to
derive the Uehling-Uhlenbeck equation, one must work with per-particle
regularity bound below $W^{4,1}$. It is certainly an interesting question to
lower the regularity requirement of Theorem \ref{THM:MainTheorem}. But we
are not able to do so currently.

Before delving into the proof of Theorem \ref{THM:MainTheorem}, we would
like to discuss the assumptions of Theorem \ref{THM:MainTheorem}. First of
all, not only we are not specifying the statistics $\left\{
f_{N}^{(k)}\right\} $ satisify, we are not assuming any statistics or
symmetric conditions on the limit $f^{(k)}$ either. Moreover, we do not need
\thinspace $f_{N}^{(k)}\left( t\right) $ or $f^{(k)}(t)$ to take a special
form, e.g. tensor product form or quasi free form, to make Theorem \ref%
{THM:MainTheorem} to hold. Compared with the work by King \cite{King} and
Landford \cite{Lanford}\footnote{%
See also \cite{SaintRaymond}.} on deriving the classical Boltzmann equation
from models with hard spheres collision and singular potentials, the
interparticle interaction $\phi $ we are considering here, is smooth, and
hence the regularity assumption in Theorem \ref{THM:MainTheorem} is not
impossible. The proof of Theorem \ref{THM:MainTheorem} suggests that the
assumption $\int \phi =0$ or $\hat{\phi}(0)=0$ might actually be a necessary
condition such that the quantum BBGKY hierarchy (\ref{hierarchy:QB-BBGKY in
Differential Form}) has a $N\rightarrow \infty $ limit. See \S \ref%
{Sec:NecessaryCondition} for a discussion. For completeness, we include, in
the appendix, a discussion about the cubic term of the Uehling-Uhlenbeck
equation (\ref{eqn:U-U}) when $\hat{\phi}(0)=0$.

\subsection{Acknowledgement}

The first author would like to thank P. Germain, E. Lieb, B. Schlein, C.
Sulem, and J. Yngvason for discussions related to this work.

\section{Proof of the Main Theorem}

For notational simplicity, it suffices to prove the main theorem for $k=1$
and with the assumption that the whole sequence $\left\{ \Gamma _{N}=\left\{
f_{N}^{(k)}\right\} _{k=1}^{N}\right\} _{N}$ has only one limit point. Our
goal is to prove the absence of cubic Uehling-Uhlenbeck terms in the limit.
Let $S^{(k)}(t)$ be the solution operator to the equation%
\begin{equation*}
\left( \partial _{t}+\mathbf{v}_{k}\cdot \nabla _{\mathbf{x}_{k}}\right)
f^{(k)}=0.
\end{equation*}%
We will prove that every limit point $\Gamma =\left\{ f^{(k)}\right\}
_{k=1}^{\infty }$ of $\left\{ \Gamma _{N}=\left\{ f_{N}^{(k)}\right\}
_{k=1}^{N}\right\} _{N}$ in the sense of Theorem \ref{THM:MainTheorem}
satisfies 
\begin{eqnarray}
&&\int J(x_{1},v_{1})f^{(1)}(t_{1},x_{1},v_{1})dx_{1}dv_{1}
\label{target:test with a test function} \\
&=&\int J(x_{1},v_{1})S^{(1)}f^{(1)}(0,x_{1},v_{1})dx_{1}dv_{1}  \notag \\
&&+\int J(x_{1},v_{1})\left(
\int_{0}^{t_{1}}S^{(1)}(t_{1}-t_{2})C_{1,2}f^{(2)}(t_{2},x_{1},v_{1})dt_{2}%
\right) dx_{1}dv_{1},  \notag
\end{eqnarray}%
for all real test function $J(x_{1},v_{1})$.

To this end, we use the BBGKY hierarchy $\left( \ref{hierarchy:QB-BBGKY in
Differential Form}\right) .$Write hierarchy $\left( \ref{hierarchy:QB-BBGKY
in Differential Form}\right) $ in integral form, we have%
\begin{eqnarray}
f_{N}^{(k)}\left( t_{k}\right) &=&S^{(k)}\left( t_{k}\right)
f_{N}^{(k)}\left( 0\right) +\int_{0}^{t_{k}}S^{(k)}(t_{k}-t_{k+1})\frac{1}{%
\sqrt{\varepsilon }}A_{\varepsilon }^{(k)}f_{N}^{(k)}(t_{k+1})dt_{k+1}
\label{hierarchy:QB-BBGKY in Integral Form} \\
&&+\int_{0}^{t_{k}}S^{(k)}(t_{k}-t_{k+1})\frac{N}{\sqrt{\varepsilon }}%
B_{\varepsilon }^{(k+1)}f_{N}^{(k+1)}(t_{k+1})dt_{k+1}.  \notag
\end{eqnarray}%
Iterate hierarchy (\ref{hierarchy:QB-BBGKY in Integral Form}) once and get to%
\begin{equation}
f_{N}^{(1)}\left( t_{1}\right) =I+II+III+IV+V,
\label{hierarchy:interateBBGKYonce}
\end{equation}%
where%
\begin{equation*}
I=S^{(1)}\left( t_{1}\right) f_{N}^{(1)}\left( 0\right) ,
\end{equation*}%
\begin{equation*}
II=\frac{1}{\sqrt{\varepsilon }}\int_{0}^{t_{1}}S^{(1)}(t_{1}-t_{2})A_{%
\varepsilon }^{(1)}f_{N}^{(1)}(t_{2})dt_{2},
\end{equation*}%
\begin{equation*}
III=\frac{N}{\sqrt{\varepsilon }}\int_{0}^{t_{1}}S^{(1)}(t_{1}-t_{2})B_{%
\varepsilon }^{(2)}S^{(2)}\left( t_{2}\right) f_{N}^{(2)}\left( 0\right)
dt_{2},
\end{equation*}%
\begin{equation*}
IV=\frac{N}{\varepsilon }\int_{0}^{t_{1}}S^{(1)}(t_{1}-t_{2})B_{\varepsilon
}^{(2)}\int_{0}^{t_{2}}S^{(2)}(t_{2}-t_{3})A_{\varepsilon
}^{(2)}f_{N}^{(2)}(t_{3})dt_{3}dt_{2},
\end{equation*}%
\begin{equation*}
V=\frac{N^{2}}{\varepsilon }\int_{0}^{t_{1}}S^{(1)}(t_{1}-t_{2})B_{%
\varepsilon }^{(2)}\int_{0}^{t_{2}}S^{(2)}(t_{2}-t_{3})B_{\varepsilon
}^{(3)}f_{N}^{(3)}(t_{3})dt_{3}dt_{2}.
\end{equation*}

On the one hand, iterating hierarchy (\ref{hierarchy:QB-BBGKY in Integral
Form}) once gives the terms which are quadratic in $\phi $ and hence are the
central part of the quantum Boltzmann hierarchy (\ref{hierarchy:QBHierarchy
in differential form}) and the Uehling-Uhlenbeck hierarchy (\ref%
{hierarchy:U-U}). On the other hand, we remark that one will not obtain the
infinite Uehling-Uhlenbeck hierarchy (\ref{hierarchy:U-U}) corresponding to
the Uehling-Uhlenbeck equation (\ref{eqn:U-U}) even one iterates (\ref%
{hierarchy:QB-BBGKY in Integral Form}) more than once and then considers its
limit as $\varepsilon \rightarrow 0$. The easiest way to see this is to
notice that the new terms will not be quadratic in $\phi .$

If one believes the mean-field limit 
\begin{equation*}
f_{N}^{(k)}\left( t,\mathbf{x}_{j},\mathbf{v}_{j}\right) \sim
\dprod\limits_{j=1}^{k}f(t,x_{j},v_{j})
\end{equation*}%
where $f$ satisfies some mean-field equation, then in the $\varepsilon
\rightarrow 0$ limit, $IV$ in (\ref{hierarchy:interateBBGKYonce}) will
generate a nonlinearity which is quadratic in $f$ and $\phi $ in the
mean-field equation, and $V$ in (\ref{hierarchy:interateBBGKYonce}) will
produce a term which is cubic in $f$ and quadratic in $\phi $. With the
above discussion in mind, alert reader can immediately tell that the main
part of the proof of Theorem \ref{THM:MainTheorem} is proving that the
Boltzmann collision operator $C_{j,k+1}$ defined in (\ref{def:collision
operator}) arises as the $\varepsilon \rightarrow 0$ limit of $IV,$ and the $%
\varepsilon \rightarrow 0$ limit of $V$ is zero and thus there is no
Uehling-Uhlenbeck term in the limit.

Since $f_{N}^{(k)}\rightarrow f^{(k)}$ in the sense stated in the main
theorem (Theorem \ref{THM:MainTheorem}), we know by definition that%
\begin{eqnarray*}
\lim_{N\rightarrow \infty }\int
J(x_{1},v_{1})f_{N}^{(1)}(t_{1},x_{1},v_{1})dx_{1}dv_{1} &=&\int
J(x_{1},v_{1})f^{(1)}(t_{1},x_{1},v_{1})dx_{1}dv_{1}, \\
\lim_{N\rightarrow \infty }\int J(x_{1},v_{1})S^{(1)}\left( t_{1}\right)
f_{N}^{(1)}\left( 0,x_{1},v_{1}\right) dx_{1}dv_{1} &=&\int
J(x_{1},v_{1})S^{(1)}\left( t_{1}\right) f^{(1)}\left( 0,x_{1},v_{1}\right)
dx_{1}dv_{1}.
\end{eqnarray*}%
Moreover, it has been shown in \cite{BCEP1,BCEP4} that the terms $II$ and $%
III$ tend to zero as $\varepsilon \rightarrow 0$. We are left to prove the
emergence of the quardratic collision kernel $C_{j,k}$ from $IV$ and the
possible cubic term $V$ is in fact zero as $\varepsilon \rightarrow 0.$

\subsection{Emergence of the Quardratic Collision Kernel\label%
{Sec:QuadraticTerm}}

$IV$ is the most important term since it contributes (\ref{def:collision
operator}) in the limit. Recall $IV$%
\begin{equation*}
IV=\frac{N}{\varepsilon }\int_{0}^{t_{1}}S^{(1)}(t_{1}-t_{2})B_{1,2}^{%
\varepsilon }\int_{0}^{t_{2}}S^{(2)}(t_{2}-t_{3})A_{1,2}^{\varepsilon
}f_{N}^{(2)}(t_{3})dt_{3}dt_{2}.
\end{equation*}%
We write%
\begin{equation*}
C_{1,2}^{\varepsilon }f_{N}^{(2)}=\frac{N}{\varepsilon }B_{1,2}^{\varepsilon
}\int_{0}^{t_{2}}S^{(2)}(t_{2}-t_{3})A_{1,2}^{\varepsilon
}f_{N}^{(2)}(t_{3})dt_{3}.
\end{equation*}%
We would like to prove 
\begin{eqnarray}
&&\lim_{N\rightarrow \infty }\int J(x_{1},v_{1})\left(
\int_{0}^{t_{1}}S^{(1)}(t_{1}-t_{2})C_{1,2}^{\varepsilon
}f_{N}^{(2)}(t_{2},x_{1},v_{1})dt_{2}\right) dx_{1}dv_{1}
\label{limit:quadratic term with S} \\
&=&\int J(x_{1},v_{1})\left(
\int_{0}^{t_{1}}S^{(1)}(t_{1}-t_{2})C_{1,2}f^{(2)}(t_{2},x_{1},v_{1})dt_{2}%
\right) dx_{1}dv_{1},  \notag
\end{eqnarray}%
and hence obtain the quardratic collision kernel which is the rightmost term
in (\ref{target:test with a test function}). Notice that $%
S^{(1)}(t_{2}-t_{1})J(x_{1},v_{1})$ is simply another test function for all $%
t_{1}$ and $t_{2}$. Hence, to establish (\ref{limit:quadratic term with S}),
it suffices to prove the following proposition.

\begin{proposition}
\label{Proposition:QuadraticTerm}Under the assumptions in Theorem \ref%
{THM:MainTheorem}, we have 
\begin{equation*}
\lim_{\varepsilon \rightarrow 0}\int J(x_{1},v_{1})C_{1,2}^{\varepsilon
}f_{N}^{(2)}(t_{2},x_{1},v_{1})dx_{1}dv_{1}=\int
J(x_{1},v_{1})C_{1,2}f^{(2)}(t_{2},x_{1},v_{1})dx_{1}dv_{1}.
\end{equation*}
\end{proposition}

\begin{proof}
We prove the propopsed limit with a direct computation. We start by writing
out $C_{1,2}^{\varepsilon }f_{N}^{(2)}$ step by step. First,%
\begin{eqnarray*}
&&\int_{0}^{t_{2}}S^{(2)}(t_{2}-t_{3})A_{1,2}^{\varepsilon
}f_{N}^{(2)}(t_{3})dt_{3} \\
&=&\frac{\left( -i\right) }{\left( 2\pi \right) ^{3}}\sum_{\sigma _{2}=\pm
1}\sigma _{2}\int_{0}^{t_{2}}dt_{3}\int_{\mathbb{R}^{3}}dh_{2} \\
&&S^{(2)}(t_{2}-t_{3})e^{\frac{ih_{2}\cdot (x_{1}-x_{2})}{\varepsilon }}\hat{%
\phi}(h_{2})f_{N}^{(2)}\left( t_{3},x_{1},x_{2},v_{1}-\sigma _{2}\frac{h_{2}%
}{2},v_{2}+\sigma _{2}\frac{h_{2}}{2}\right) \\
&=&\frac{\left( -i\right) }{\left( 2\pi \right) ^{3}}\sum_{\sigma _{2}=\pm
1}\sigma _{2}\int_{0}^{t_{2}}dt_{3}\int_{\mathbb{R}^{3}}dh_{2}e^{\frac{%
ih_{2}\cdot (x_{1}-\left( t_{2}-t_{3}\right) v_{1}-x_{2}+\left(
t_{2}-t_{3}\right) v_{2})}{\varepsilon }}\hat{\phi}(h_{2}) \\
&&f_{N}^{(2)}\left( t_{3},x_{1}-\left( t_{2}-t_{3}\right) v_{1},x_{2}-\left(
t_{2}-t_{3}\right) v_{2},v_{1}-\sigma _{2}\frac{h_{2}}{2},v_{2}+\sigma _{2}%
\frac{h_{2}}{2}\right)
\end{eqnarray*}%
then%
\begin{eqnarray*}
&&B_{1,2}^{\varepsilon
}\int_{0}^{t_{2}}S^{(2)}(t_{2}-t_{3})A_{1,2}^{\varepsilon
}f^{(2)}(t_{3})dt_{3} \\
&=&\frac{\left( -i\right) ^{2}}{\left( 2\pi \right) ^{6}}\sum_{\sigma
_{1},\sigma _{2}=\pm 1}\sigma _{1}\sigma _{2}\int_{\mathbb{R}%
^{3}}dx_{2}\int_{\mathbb{R}^{3}}dv_{2}\int_{\mathbb{R}^{3}}dh_{1}e^{\frac{%
ih_{1}\cdot (x_{1}-x_{2})}{\varepsilon }}\hat{\phi}(h_{1}) \\
&&\int_{0}^{t_{2}}dt_{3}\int_{\mathbb{R}^{3}}dh_{2}e^{\frac{ih_{2}\cdot %
\left[ x_{1}-\left( t_{2}-t_{3}\right) \left( v_{1}-\sigma _{1}\frac{h_{1}}{2%
}\right) -x_{2}+\left( t_{2}-t_{3}\right) \left( v_{2}+\sigma _{1}\frac{h_{1}%
}{2}\right) \right] }{\varepsilon }}\hat{\phi}(h_{2}) \\
&&f_{N}^{(2)}(t_{3},x_{1}-\left( t_{2}-t_{3}\right) \left( v_{1}-\sigma _{1}%
\frac{h_{1}}{2}\right) ,x_{2}-\left( t_{2}-t_{3}\right) \left( v_{2}+\sigma
_{1}\frac{h}{2}\right) , \\
&&v_{1}-\sigma _{2}\frac{h_{2}}{2}-\sigma _{1}\frac{h_{1}}{2},v_{2}+\sigma
_{2}\frac{h_{2}}{2}+\sigma _{1}\frac{h_{1}}{2}).
\end{eqnarray*}%
Rearrange, we have%
\begin{eqnarray*}
&=&\frac{\left( -i\right) ^{2}}{\left( 2\pi \right) ^{6}}\sum_{\sigma
_{1},\sigma _{2}=\pm 1}\sigma _{1}\sigma _{2}\int_{\mathbb{R}%
^{3}}dx_{2}\int_{\mathbb{R}^{3}}dv_{2}\int_{0}^{t_{2}}dt_{3}\int_{\mathbb{R}%
^{3}}dh_{1}\int_{\mathbb{R}^{3}}dh_{2} \\
&&e^{\frac{ih_{1}\cdot (x_{1}-x_{2})}{\varepsilon }}e^{\frac{ih_{2}\cdot %
\left[ x_{1}-x_{2}-\left( t_{2}-t_{3}\right) \left( v_{1}-v_{2}-\sigma
_{1}h_{1}\right) \right] }{\varepsilon }}\hat{\phi}(h_{1})\hat{\phi}(h_{2})
\\
&&f_{N}^{(2)}(t_{3},x_{1}-\left( t_{2}-t_{3}\right) \left( v_{1}-\sigma _{1}%
\frac{h_{1}}{2}\right) ,x_{2}-\left( t_{2}-t_{3}\right) \left( v_{2}+\sigma
_{1}\frac{h_{1}}{2}\right) , \\
&&v_{1}-\sigma _{2}\frac{h_{2}}{2}-\sigma _{1}\frac{h_{1}}{2},v_{2}+\sigma
_{2}\frac{h_{2}}{2}+\sigma _{1}\frac{h_{1}}{2}).
\end{eqnarray*}%
So%
\begin{eqnarray*}
&&\int J(x_{1},v_{1})C_{1,2}^{\varepsilon
}f_{N}^{(2)}(t_{2},x_{1},v_{1})dx_{1}dv_{1} \\
&=&\frac{N}{\varepsilon }\frac{\left( -i\right) ^{2}}{\left( 2\pi \right)
^{6}}\sum_{\sigma _{1},\sigma _{2}=\pm 1}\sigma _{1}\sigma _{2}\int d\mathbf{%
x}_{2}\int d\mathbf{v}_{2}\int_{0}^{t_{2}}dt_{3}\int_{\mathbb{R}%
^{3}}dh_{1}\int_{\mathbb{R}^{3}}dh_{2} \\
&&J(x_{1},v_{1})e^{\frac{ih_{1}\cdot (x_{1}-x_{2})}{\varepsilon }}e^{\frac{%
ih_{2}\cdot \left[ x_{1}-x_{2}-\left( t_{2}-t_{3}\right) \left(
v_{1}-v_{2}-\sigma _{1}h_{1}\right) \right] }{\varepsilon }}\hat{\phi}(h_{1})%
\hat{\phi}(h_{2}) \\
&&f_{N}^{(2)}(t_{3},x_{1}-\left( t_{2}-t_{3}\right) \left( v_{1}-\sigma _{1}%
\frac{h_{1}}{2}\right) ,x_{2}-\left( t_{2}-t_{3}\right) \left( v_{2}+\sigma
_{1}\frac{h_{1}}{2}\right) , \\
&&v_{1}-\sigma _{2}\frac{h_{2}}{2}-\sigma _{1}\frac{h_{1}}{2},v_{2}+\sigma
_{2}\frac{h_{2}}{2}+\sigma _{1}\frac{h_{1}}{2}).
\end{eqnarray*}%
The $h_{1}$ and $h_{2}$ integrals are highly oscillatory. We change
variables to move the $h^{\prime }s$ away from $f_{N}^{(2)}:$ first the $x$
part,%
\begin{eqnarray*}
x_{1,new} &=&x_{1,old}-\left( t_{2}-t_{3}\right) \left( v_{1}-\sigma _{1}%
\frac{h_{1}}{2}\right) , \\
x_{2,new} &=&x_{2,old}-\left( t_{2}-t_{3}\right) \left( v_{2}+\sigma _{1}%
\frac{h_{1}}{2}\right) ,
\end{eqnarray*}%
which gives%
\begin{eqnarray*}
&&\int J(x_{1},v_{1})C_{1,2}^{\varepsilon
}f_{N}^{(2)}(t_{2},x_{1},v_{1})dx_{1}dv_{1} \\
&=&\frac{N}{\varepsilon }\frac{\left( -i\right) ^{2}}{\left( 2\pi \right)
^{6}}\sum_{\sigma _{1},\sigma _{2}=\pm 1}\sigma _{1}\sigma _{2}\int d\mathbf{%
x}_{2}\int d\mathbf{v}_{2}\int_{0}^{t_{2}}dt_{3}\int_{\mathbb{R}%
^{3}}dh_{1}\int_{\mathbb{R}^{3}}dh_{2} \\
&&J(x_{1}+\left( t_{2}-t_{3}\right) \left( v_{1}-\sigma _{1}\frac{h_{1}}{2}%
\right) ,v_{1}) \\
&&e^{\frac{ih_{1}\cdot (x_{1}-x_{2}+\left( t_{2}-t_{3}\right) \left(
v_{1}-v_{2}-\sigma _{1}h_{1}\right) )}{\varepsilon }}e^{\frac{ih_{2}\cdot
(x_{1}-x_{2})}{\varepsilon }}\hat{\phi}(h_{1})\hat{\phi}(h_{2}) \\
&&f_{N}^{(2)}\left( t_{3},x_{1},x_{2},v_{1}-\sigma _{2}\frac{h_{2}}{2}%
-\sigma _{1}\frac{h_{1}}{2},v_{2}+\sigma _{2}\frac{h_{2}}{2}+\sigma _{1}%
\frac{h_{1}}{2}\right) .
\end{eqnarray*}%
Then the $v$ part%
\begin{eqnarray*}
v_{1,new} &=&v_{1,old}-\sigma _{2}\frac{h_{2}}{2}-\sigma _{1}\frac{h_{1}}{2},
\\
v_{2,new} &=&v_{2,old}+\sigma _{2}\frac{h_{2}}{2}+\sigma _{1}\frac{h_{1}}{2},
\end{eqnarray*}%
which yields%
\begin{eqnarray*}
&&\int J(x_{1},v_{1})C_{1,2}^{\varepsilon
}f_{N}^{(2)}(t_{2},x_{1},v_{1})dx_{1}dv_{1} \\
&=&\frac{N}{\varepsilon }\frac{\left( -i\right) ^{2}}{\left( 2\pi \right)
^{6}}\sum_{\sigma _{1},\sigma _{2}=\pm 1}\sigma _{1}\sigma _{2}\int d\mathbf{%
x}_{2}\int d\mathbf{v}_{2}\int_{0}^{t_{2}}dt_{3}\int_{\mathbb{R}%
^{3}}dh_{1}\int_{\mathbb{R}^{3}}dh_{2} \\
&&J(x_{1}+\left( t_{2}-t_{3}\right) \left( v_{1}+\sigma _{1}\frac{h_{2}}{2}%
\right) ,v_{1}+\sigma _{2}\frac{h_{2}}{2}+\sigma _{1}\frac{h_{1}}{2}) \\
&&e^{\frac{ih_{1}\cdot (x_{1}-x_{2}+\left( t_{2}-t_{3}\right) \left(
v_{1}-v_{2}+\sigma _{2}h_{2}\right) )}{\varepsilon }}e^{\frac{ih_{2}\cdot
(x_{1}-x_{2})}{\varepsilon }}\hat{\phi}(h_{1})\hat{\phi}(h_{2})f_{N}^{(2)}%
\left( t_{3},x_{1},x_{2},v_{1},v_{2}\right) .
\end{eqnarray*}

To evaluate the above integral, we substitute like \cite[(2.15)]{BCEP1}%
\begin{equation*}
t_{3}=t_{2}-\varepsilon s_{1},\text{ \ \ \ \ \ \ \ }h_{1}=\varepsilon \xi
_{1}-h_{2},
\end{equation*}%
and have%
\begin{eqnarray*}
&&\int J(x_{1},v_{1})C_{1,2}^{\varepsilon
}f_{N}^{(2)}(t_{2},x_{1},v_{1})dx_{1}dv_{1} \\
&=&\frac{\varepsilon ^{4}}{\varepsilon ^{4}}\frac{\left( -i\right) ^{2}}{%
\left( 2\pi \right) ^{6}}\sum_{\sigma _{1},\sigma _{2}=\pm 1}\sigma
_{1}\sigma _{2}\int d\mathbf{x}_{2}\int d\mathbf{v}_{2}\int_{0}^{\frac{t_{2}%
}{\varepsilon }}ds_{1}\int_{\mathbb{R}^{3}}d\xi _{1}\int_{\mathbb{R}%
^{3}}dh_{2} \\
&&J(x_{1}+\varepsilon s_{1}\left( v_{1}+\sigma _{1}\frac{h_{2}}{2}\right)
,v_{1}+\sigma _{2}\frac{h_{2}}{2}+\sigma _{1}\frac{\left( \varepsilon \xi
_{1}-h_{2}\right) }{2}) \\
&&e^{\frac{i\left( \varepsilon \xi _{1}-h_{2}\right) \cdot
(x_{1}-x_{2}+\varepsilon s_{1}\left( v_{1}-v_{2}+\sigma _{2}h_{2}\right) )}{%
\varepsilon }}e^{\frac{ih_{2}\cdot (x_{1}-x_{2})}{\varepsilon }}\hat{\phi}%
(\varepsilon \xi _{1}-h_{2})\hat{\phi}(h_{2})f_{N}^{(2)}\left(
t_{2}-\varepsilon s_{1},x_{1},x_{2},v_{1},v_{2}\right)
\end{eqnarray*}%
which simplifies to%
\begin{eqnarray}
&=&\frac{\left( -i\right) ^{2}}{\left( 2\pi \right) ^{6}}\sum_{\sigma
_{1},\sigma _{2}=\pm 1}\sigma _{1}\sigma _{2}\int d\mathbf{x}_{2}\int d%
\mathbf{v}_{2}\int_{0}^{\frac{t_{2}}{\varepsilon }}ds_{1}\int_{\mathbb{R}%
^{3}}d\xi _{1}\int_{\mathbb{R}^{3}}dh_{2}
\label{eqn:quadratic term before limit} \\
&&J(x_{1}+\varepsilon s_{1}\left( v_{1}+\sigma _{1}\frac{h_{2}}{2}\right)
,v_{1}+\sigma _{2}\frac{h_{2}}{2}+\sigma _{1}\frac{\left( \varepsilon \xi
_{1}-h_{2}\right) }{2})\hat{\phi}(\varepsilon \xi _{1}-h_{2})\hat{\phi}%
(h_{2})  \notag \\
&&e^{i\xi _{1}\cdot \left( x_{1}-x_{2}\right) }e^{i\left( \varepsilon \xi
_{1}-h_{2}\right) \cdot s_{1}\left( v_{1}-v_{2}+\sigma _{2}h_{2}\right)
}f_{N}^{(2)}\left( t_{2}-\varepsilon s_{1},x_{1},x_{2},v_{1},v_{2}\right) . 
\notag
\end{eqnarray}%
Taking the $\varepsilon \rightarrow 0$ limit (justified in \S \ref%
{Sec:Interchange}), we arrive at 
\begin{eqnarray}
&&\lim_{\varepsilon \rightarrow 0}\int J(x_{1},v_{1})C_{1,2}^{\varepsilon
}f_{N}^{(2)}(t_{2},x_{1},v_{1})dx_{1}dv_{1}
\label{eqn:quardratic term after limit} \\
&=&\frac{\left( -i\right) ^{2}}{\left( 2\pi \right) ^{6}}\sum_{\sigma
_{1},\sigma _{2}=\pm 1}\sigma _{1}\sigma _{2}\int d\mathbf{x}_{2}\int d%
\mathbf{v}_{2}\int_{0}^{+\infty }ds_{1}\int_{\mathbb{R}^{3}}d\xi _{1}\int_{%
\mathbb{R}^{3}}dh_{2}  \notag \\
&&J(x_{1},v_{1}+\sigma _{2}\frac{h_{2}}{2}-\sigma _{1}\frac{h_{2}}{2})\hat{%
\phi}(-h_{2})\hat{\phi}(h_{2})  \notag \\
&&e^{i\xi _{1}\cdot \left( x_{1}-x_{2}\right) }e^{-ih_{2}\cdot s_{1}\left(
v_{1}-v_{2}+\sigma _{2}h_{2}\right) }f^{(2)}\left(
t_{2},x_{1},x_{2},v_{1},v_{2}\right)  \notag
\end{eqnarray}%
Using the fact that 
\begin{equation*}
\int_{\mathbb{R}^{3}}e^{i\xi _{1}\cdot \left( x_{1}-x_{2}\right) }d\xi
_{1}=\left( 2\pi \right) ^{3}\delta (x_{1}-x_{2}),
\end{equation*}%
(\ref{eqn:quardratic term after limit}) becomes 
\begin{eqnarray*}
&=&\frac{\left( -i\right) ^{2}}{\left( 2\pi \right) ^{3}}\sum_{\sigma
_{1},\sigma _{2}=\pm 1}\sigma _{1}\sigma _{2}\int d\mathbf{x}_{2}\int d%
\mathbf{v}_{2}\int_{0}^{+\infty }ds_{1}\int_{\mathbb{R}^{3}}dh_{2} \\
&&J(x_{1},v_{1}+\sigma _{2}\frac{h_{2}}{2}-\sigma _{1}\frac{h_{2}}{2})\hat{%
\phi}(-h_{2})\hat{\phi}(h_{2})\delta (x_{1}-x_{2}) \\
&&e^{-ih_{2}\cdot s_{1}\left( v_{1}-v_{2}+\sigma _{2}h_{2}\right)
}f^{(2)}\left( t_{2},x_{1},x_{2},v_{1},v_{2}\right) \\
&=&\frac{\left( -i\right) ^{2}}{\left( 2\pi \right) ^{3}}\sum_{\sigma
_{1},\sigma _{2}=\pm 1}\sigma _{1}\sigma _{2}\int dx_{1}\int d\mathbf{v}%
_{2}\int_{0}^{+\infty }ds_{1}\int_{\mathbb{R}^{3}}dh_{2} \\
&&J(x_{1},v_{1}+\sigma _{2}\frac{h_{2}}{2}-\sigma _{1}\frac{h_{2}}{2})\hat{%
\phi}(-h_{2})\hat{\phi}(h_{2})e^{-ih_{2}\cdot s_{1}\left( v_{1}-v_{2}+\sigma
_{2}h_{2}\right) }f^{(2)}\left( t_{2},x_{1},x_{1},v_{1},v_{2}\right)
\end{eqnarray*}%
Put in spherical coordinates for the $dh_{2}$ integration: we let $%
h_{2}=r\omega $, where $r\in \mathbb{R}^{+}$ and $\omega \in \mathbb{S}^{2},$
to get%
\begin{eqnarray*}
&&\frac{\left( -i\right) ^{2}}{\left( 2\pi \right) ^{3}}\sum_{\sigma
_{1},\sigma _{2}=\pm 1}\sigma _{1}\sigma _{2}\int dx_{1}\int d\mathbf{v}%
_{2}\int_{0}^{+\infty }ds_{1}\int_{0}^{\infty }r^{2}dr\int_{\mathbb{S}%
^{2}}dS_{\omega } \\
&&J(x_{1},v_{1}+\sigma _{2}\frac{r\omega }{2}-\sigma _{1}\frac{r\omega }{2}%
)\left\vert \hat{\phi}(r\omega )\right\vert ^{2}e^{-ir\omega \cdot
s_{1}\left( v_{1}-v_{2}+\sigma _{2}r\omega \right) }f^{(2)}\left(
t_{2},x_{1},x_{1},v_{1},v_{2}\right) .
\end{eqnarray*}%
Substitute $u=rs_{1}$, 
\begin{eqnarray}
&=&\frac{\left( -i\right) ^{2}}{\left( 2\pi \right) ^{3}}\sum_{\sigma
_{1},\sigma _{2}=\pm 1}\sigma _{1}\sigma _{2}\int dx_{1}\int d\mathbf{v}%
_{2}\int_{0}^{+\infty }du  \label{eqn:simplified quardratic term after limit}
\\
&&\int_{0}^{\infty }rdr\int_{\mathbb{S}^{2}}dS_{\omega }J(x_{1},v_{1}+\sigma
_{2}\frac{r\omega }{2}-\sigma _{1}\frac{r\omega }{2})\left\vert \hat{\phi}%
(r\omega )\right\vert ^{2}  \notag \\
&&e^{-iu\left[ \left( v_{1}-v_{2}\right) \cdot \omega +\sigma _{2}r\right]
}f^{(2)}\left( t_{2},x_{1},x_{1},v_{1},v_{2}\right) .  \notag
\end{eqnarray}

Write $J(x_{1},v_{1}+\sigma _{2}\frac{r\omega }{2}-\sigma _{1}\frac{r\omega 
}{2})=g(\left( \sigma _{2}-\sigma _{1}\right) \frac{r\omega }{2})$ for short
at the moment. Notice that in the middle of (\ref{eqn:simplified quardratic
term after limit}), we have%
\begin{eqnarray}
&&\sum_{\sigma _{1},\sigma _{2}=\pm 1}\sigma _{1}\sigma _{2}\left(
\int_{0}^{+\infty }du\int_{\mathbb{S}^{2}}dS_{\omega }g(\left( \sigma
_{2}-\sigma _{1}\right) \frac{r\omega }{2})\left\vert \hat{\phi}(r\omega
)\right\vert ^{2}e^{-iu\left[ \left( v_{1}-v_{2}\right) \cdot \omega +\sigma
_{2}r\right] }\right)  \label{eqn:surface measure} \\
&=&\left( \int_{0}^{+\infty }du\int_{\mathbb{S}^{2}}dS_{\omega
}g(0)\left\vert \hat{\phi}(r\omega )\right\vert ^{2}e^{-iu\left[ \left(
v_{1}-v_{2}\right) \cdot \omega +r\right] }\right)  \notag \\
&&-\left( \int_{0}^{+\infty }du\int_{\mathbb{S}^{2}}dS_{\omega }g(r\omega
)\left\vert \hat{\phi}(r\omega )\right\vert ^{2}e^{-iu\left[ \left(
v_{1}-v_{2}\right) \cdot \omega +r\right] }\right)  \notag \\
&&+\left( \int_{0}^{+\infty }du\int_{\mathbb{S}^{2}}dS_{\omega
}g(0)\left\vert \hat{\phi}(r\omega )\right\vert ^{2}e^{-iu\left[ \left(
v_{1}-v_{2}\right) \cdot \omega -r\right] }\right)  \notag \\
&&-\left( \int_{0}^{+\infty }du\int_{\mathbb{S}^{2}}dS_{\omega }g(-r\omega
)\left\vert \hat{\phi}(r\omega )\right\vert ^{2}e^{-iu\left[ \left(
v_{1}-v_{2}\right) \cdot \omega -r\right] }\right)  \notag \\
&=&A-B+C-D.  \notag
\end{eqnarray}%
Do the substitution, $\omega _{new}=-\omega _{old}$ in terms $C$ and $D$, we
then find that $C=\bar{A}$ and $D=\bar{B}.$ So (\ref{eqn:surface measure})
is actually%
\begin{eqnarray*}
&=&\int_{\mathbb{S}^{2}}dS_{\omega }\left( g(0)\left\vert \hat{\phi}(r\omega
)\right\vert ^{2}2\func{Re}\int_{0}^{+\infty }due^{-iu\left[ \left(
v_{1}-v_{2}\right) \cdot \omega +r\right] }\right) \\
&&-\int_{\mathbb{S}^{2}}dS_{\omega }\left( g(r\omega )\left\vert \hat{\phi}%
(r\omega )\right\vert ^{2}2\func{Re}\int_{0}^{+\infty }due^{-iu\left[ \left(
v_{1}-v_{2}\right) \cdot \omega +r\right] }\right) ,
\end{eqnarray*}%
where%
\begin{eqnarray*}
&&\func{Re}\int_{0}^{+\infty }due^{-iu\left[ \left( v_{1}-v_{2}\right) \cdot
\omega +r\right] }=\func{Re}\int_{-\infty }^{+\infty }due^{-iu\left[ \left(
v_{1}-v_{2}\right) \cdot \omega +r\right] }H(u) \\
&=&\func{Re}\hat{H}(\left[ \left( v_{1}-v_{2}\right) \cdot \omega +r\right]
)=\pi \delta \left( \left( v_{1}-v_{2}\right) \cdot \omega +r\right) .
\end{eqnarray*}%
if we denote the Heaviside function by $H$.

Putting the above computation of (\ref{eqn:surface measure}) into (\ref%
{eqn:simplified quardratic term after limit}), we have%
\begin{eqnarray*}
&&\lim_{\varepsilon \rightarrow 0}\int J(x_{1},v_{1})C_{1,2}^{\varepsilon
}f_{N}^{(2)}(t_{2},x_{1},v_{1})dx_{1}dv_{1} \\
&=&\frac{-2}{8\pi ^{2}}\int dx_{1}\int d\mathbf{v}_{2}\int_{0}^{\infty
}rdr\int_{\mathbb{S}^{2}}dS_{\omega }\left[ J(x_{1},v_{1})-J(x_{1},v_{1}+r%
\omega )\right] \\
&&\delta \left( \left( v_{1}-v_{2}\right) \cdot \omega +r\right) \left\vert 
\hat{\phi}(r\omega )\right\vert ^{2}f^{(2)}\left(
t_{2},x_{1},x_{1},v_{1},v_{2}\right)
\end{eqnarray*}%
Insert a Heaviside function $H(r)$ to do the $dr$ integral, 
\begin{eqnarray*}
&=&\frac{2}{8\pi ^{2}}\int dx_{1}\int d\mathbf{v}_{2}\int_{-\infty }^{\infty
}H(r)rdr\int_{\mathbb{S}^{2}}dS_{\omega } \\
&&\left( J(x_{1},v_{1}+r\omega )-J(x_{1},v_{1})\right) \\
&&\pi \delta \left( \left( v_{1}-v_{2}\right) \cdot \omega +r\right)
\left\vert \hat{\phi}(r\omega )\right\vert ^{2}f^{(2)}\left(
t_{2},x_{1},x_{1},v_{1},v_{2}\right) \\
&=&\frac{2}{8\pi ^{2}}\int dx_{1}\int d\mathbf{v}_{2}\int_{\mathbb{S}%
^{2}}dS_{\omega }\left( J(x_{1},v_{1}-\left[ \left( v_{1}-v_{2}\right) \cdot
\omega \right] \omega )-J(x_{1},v_{1})\right) \\
&&H(-\left( v_{1}-v_{2}\right) \cdot \omega )\left( -\left(
v_{1}-v_{2}\right) \cdot \omega \right) \left\vert \hat{\phi}(\left[ \left(
v_{1}-v_{2}\right) \cdot \omega \right] \omega )\right\vert
^{2}f^{(2)}\left( t_{2},x_{1},x_{1},v_{1},v_{2}\right)
\end{eqnarray*}%
That is%
\begin{eqnarray*}
&=&\frac{1}{8\pi ^{2}}\int dx_{1}\int d\mathbf{v}_{2}\int_{\mathbb{S}%
^{2}}dS_{\omega }\left( J(x_{1},v_{1}-\left[ \left( v_{1}-v_{2}\right) \cdot
\omega \right] \omega )-J(x_{1},v_{1})\right) \\
&&\left\vert \left( v_{1}-v_{2}\right) \cdot \omega \right\vert \left\vert 
\hat{\phi}(\left[ \left( v_{1}-v_{2}\right) \cdot \omega \right] \omega
)\right\vert ^{2}f^{(2)}\left( t_{2},x_{1},x_{1},v_{1},v_{2}\right) ,
\end{eqnarray*}%
which is exactly%
\begin{equation*}
\int J(x_{1},v_{1})C_{1,2}f^{(2)}(t_{2},x_{1},v_{1})dx_{1}dv_{1}.
\end{equation*}%
Whence we conclude the proof of Proposition \ref{Proposition:QuadraticTerm}.
\end{proof}

\subsection{The Cubic Term is Zero\label{Sec:CubicTerm}}

Here, we investigate the limit of 
\begin{equation*}
V=\frac{N^{2}}{\varepsilon }\int_{0}^{t_{1}}S^{(1)}(t_{1}-t_{2})B_{%
\varepsilon }^{(2)}\int_{0}^{t_{2}}S^{(2)}(t_{2}-t_{3})B_{\varepsilon
}^{(3)}f_{N}^{(3)}(t_{3})dt_{3}dt_{2}.
\end{equation*}%
We write%
\begin{eqnarray*}
Q_{1,3}^{\varepsilon }f_{N}^{(3)} &=&\frac{N^{2}}{\varepsilon }%
B_{1,2}^{\varepsilon }\int_{0}^{t_{2}}S^{(2)}(t_{2}-t_{3})\left(
B_{1,3}^{\varepsilon }+B_{2,3}^{\varepsilon }\right) f_{N}^{(3)}(t_{3})dt_{3}
\\
&=&Q_{1,3}^{\varepsilon ,1}f_{N}^{(3)}+Q_{1,3}^{\varepsilon ,2}f_{N}^{(3)}.
\end{eqnarray*}%
If the $\varepsilon \rightarrow 0$ limit of $Q_{1,3}^{\varepsilon
}f_{N}^{(3)}$ is nonzero, it will correspond to a cubic nonlinearity in the
mean-field equation. On the one hand, as we remarked earlier in the paper,
for the Uehling-Uhlenbeck equation (\ref{eqn:U-U}) to rise as the mean-field
equation, $\lim_{\varepsilon \rightarrow 0}Q_{1,3}^{\varepsilon }f_{N}^{(3)}$
must not be zero. On the other hand, $\lim_{\varepsilon \rightarrow
0}Q_{1,3}^{\varepsilon }f_{N}^{(3)}$ has to be zero for Theorem \ref%
{THM:MainTheorem} to hold. Hence, we compute $\lim_{\varepsilon \rightarrow
0}Q_{1,3}^{\varepsilon ,1}f_{N}^{(3)}$ and $\lim_{\varepsilon \rightarrow
0}Q_{1,3}^{\varepsilon ,2}f_{N}^{(3)}$ in complete detail.

\subsubsection{Treatment of $Q_{1,3}^{\protect\varepsilon ,1}f_{N}^{(3)}$}

We prove that the limit%
\begin{equation*}
\lim_{\varepsilon \rightarrow 0}\int J(x_{1},v_{1})Q_{1,3}^{\varepsilon
,1}f_{N}^{(3)}(t_{2})dx_{1}dv_{1}=0
\end{equation*}%
by direct computation. Since the proposed limit is zero, we drop the
prefactor $\frac{\left( -i\right) }{\left( 2\pi \right) ^{3}}$ in $%
B_{\varepsilon }$ so that we do not need to keep track of it. We write%
\begin{eqnarray*}
&&\int_{0}^{t_{2}}S^{(2)}(t_{2}-t_{3})B_{1,3}^{\varepsilon
}f_{N}^{(3)}(t_{3})dt_{3} \\
&=&\sum_{\sigma _{2}=\pm 1}\sigma _{2}\int_{0}^{t_{2}}dt_{3}\int_{\mathbb{R}%
^{3}}dx_{3}\int_{\mathbb{R}^{3}}dv_{3}\int_{\mathbb{R}^{3}}dh_{2} \\
&&S^{(2)}(t_{2}-t_{3})e^{\frac{ih_{2}\cdot (x_{1}-x_{3})}{\varepsilon }}\hat{%
\phi}(h_{2})f_{N}^{(3)}\left( t_{3},x_{1},x_{2},x_{3},v_{1}-\sigma _{2}\frac{%
h_{2}}{2},v_{2},v_{3}+\sigma _{2}\frac{h_{2}}{2}\right) .
\end{eqnarray*}%
Different from the quadratic term treated in \S \ref{Sec:QuadraticTerm}, $%
S^{(2)}$ has no effect on $(x_{3},v_{3})$, so it becomes 
\begin{eqnarray*}
&=&\sum_{\sigma _{2}=\pm 1}\sigma _{2}\int_{0}^{t_{2}}dt_{3}\int_{\mathbb{R}%
^{3}}dx_{3}\int_{\mathbb{R}^{3}}dv_{3}\int_{\mathbb{R}^{3}}dh_{2}e^{\frac{%
ih_{2}\cdot \left[ \left( x_{1}-(t_{2}-t_{3})v_{1}\right) -x_{3}\right] }{%
\varepsilon }}\hat{\phi}(h_{2}) \\
&&\times f_{N}^{(3)}\left(
t_{3},x_{1}-(t_{2}-t_{3})v_{1},x_{2}-(t_{2}-t_{3})v_{2},x_{3},v_{1}-\sigma
_{2}\frac{h_{2}}{2},v_{2},v_{3}+\sigma _{2}\frac{h_{2}}{2}\right) .
\end{eqnarray*}%
Then%
\begin{eqnarray*}
&&\frac{N^{2}}{\varepsilon }B_{1,2}^{\varepsilon
}\int_{0}^{t_{2}}S^{(2)}(t_{2}-t_{3})B_{1,3}^{\varepsilon
}f_{N}^{(3)}(t_{3})dt_{3} \\
&=&\frac{N^{2}}{\varepsilon }\sum_{\sigma _{1},\sigma _{2}=\pm 1}\sigma
_{1}\sigma _{2}\int_{0}^{t_{2}}dt_{3}\int_{\mathbb{R}^{3}}dx_{2}\int_{%
\mathbb{R}^{3}}dx_{3}\int_{\mathbb{R}^{3}}dv_{2}\int_{\mathbb{R}%
^{3}}dv_{3}\int_{\mathbb{R}^{3}}dh_{1}\int_{\mathbb{R}^{3}}dh_{2} \\
&&e^{\frac{ih_{1}\cdot (x_{1}-x_{2})}{\varepsilon }}e^{\frac{ih_{2}\cdot %
\left[ \left( x_{1}-(t_{2}-t_{3})\left( v_{1}-\sigma _{1}\frac{h_{1}}{2}%
\right) \right) -x_{3}\right] }{\varepsilon }}\hat{\phi}(h_{1})\hat{\phi}%
(h_{2}) \\
&&\times f_{N}^{(3)}(t_{3},x_{1}-(t_{2}-t_{3})\left( v_{1}-\sigma _{1}\frac{%
h_{1}}{2}\right) ,x_{2}-(t_{2}-t_{3})\left( v_{2}+\sigma _{1}\frac{h_{1}}{2}%
\right) ,x_{3}, \\
&&v_{1}-\sigma _{2}\frac{h_{2}}{2}-\sigma _{1}\frac{h_{1}}{2},v_{2}+\sigma
_{1}\frac{h_{1}}{2},v_{3}+\sigma _{2}\frac{h_{2}}{2}),
\end{eqnarray*}%
thus%
\begin{eqnarray*}
&&\int J(x_{1},v_{1})Q_{1,3}^{\varepsilon ,1}f_{N}^{(3)}dx_{1}dv_{1} \\
&=&\frac{1}{\varepsilon ^{7}}\sum_{\sigma _{1},\sigma _{2}=\pm 1}\sigma
_{1}\sigma _{2}\int d\mathbf{x}_{3}\int d\mathbf{v}_{3}%
\int_{0}^{t_{2}}dt_{3}\int_{\mathbb{R}^{3}}dh_{1}\int_{\mathbb{R}^{3}}dh_{2}
\\
&&e^{\frac{ih_{1}\cdot (x_{1}-x_{2})}{\varepsilon }}e^{\frac{ih_{2}\cdot %
\left[ \left( x_{1}-(t_{2}-t_{3})\left( v_{1}-\sigma _{1}\frac{h_{1}}{2}%
\right) \right) -x_{3}\right] }{\varepsilon }}\hat{\phi}(h_{1})\hat{\phi}%
(h_{2})J(x_{1},v_{1}) \\
&&\times f_{N}^{(3)}(t_{3},x_{1}-(t_{2}-t_{3})\left( v_{1}-\sigma _{1}\frac{%
h_{1}}{2}\right) ,x_{2}-(t_{2}-t_{3})\left( v_{2}+\sigma _{1}\frac{h_{1}}{2}%
\right) ,x_{3}, \\
&&v_{1}-\sigma _{2}\frac{h_{2}}{2}-\sigma _{1}\frac{h_{1}}{2},v_{2}+\sigma
_{1}\frac{h_{1}}{2},v_{3}+\sigma _{2}\frac{h_{2}}{2}).
\end{eqnarray*}%
We move all $h$'s away from $f_{N}^{(3)}$. First, we substitute the $x$-part
with%
\begin{eqnarray*}
x_{1,new} &=&x_{1,old}-(t_{2}-t_{3})\left( v_{1}-\sigma _{1}\frac{h_{1}}{2}%
\right) \\
x_{2,new} &=&x_{2,old}-(t_{2}-t_{3})\left( v_{2}+\sigma _{1}\frac{h_{1}}{2}%
\right)
\end{eqnarray*}%
which gives 
\begin{eqnarray*}
&&\int J(x_{1},v_{1})Q_{1,3}^{\varepsilon ,1}f_{N}^{(3)}dx_{1}dv_{1} \\
&=&\frac{1}{\varepsilon ^{7}}\sum_{\sigma _{1},\sigma _{2}=\pm 1}\sigma
_{1}\sigma _{2}\int d\mathbf{x}_{3}\int d\mathbf{v}_{3}%
\int_{0}^{t_{2}}dt_{3}\int_{\mathbb{R}^{3}}dh_{1}\int_{\mathbb{R}^{3}}dh_{2}
\\
&&e^{\frac{ih_{1}\cdot (x_{1}+(t_{2}-t_{3})\left( v_{1}-\sigma _{1}\frac{%
h_{1}}{2}\right) -x_{2}-(t_{2}-t_{3})\left( v_{2}+\sigma _{1}\frac{h_{1}}{2}%
\right) )}{\varepsilon }}e^{\frac{ih_{2}\cdot \left( x_{1}-x_{3}\right) }{%
\varepsilon }}\hat{\phi}(h_{1})\hat{\phi}(h_{2}) \\
&&J(x_{1}+(t_{2}-t_{3})\left( v_{1}-\sigma _{1}\frac{h_{1}}{2}\right) ,v_{1})
\\
&&f_{N}^{(3)}(t_{3},x_{1},x_{2},x_{3},v_{1}-\sigma _{2}\frac{h_{2}}{2}%
-\sigma _{1}\frac{h_{1}}{2},v_{2}+\sigma _{1}\frac{h_{1}}{2},v_{3}+\sigma
_{2}\frac{h_{2}}{2}).
\end{eqnarray*}%
Then the $v$-substitution: 
\begin{equation*}
v_{1,new}=v_{1,old}-\sigma _{2}\frac{h_{2}}{2}-\sigma _{1}\frac{h_{1}}{2},%
\text{ \ \ }v_{2,new}=v_{2,old}+\sigma _{1}\frac{h_{1}}{2},\text{ \ \ }%
v_{3,new}=v_{3,old}+\sigma _{2}\frac{h_{2}}{2}
\end{equation*}%
gives%
\begin{eqnarray*}
&&\int J(x_{1},v_{1})Q_{1,3}^{\varepsilon ,1}f_{N}^{(3)}dx_{1}dv_{1} \\
&=&\frac{1}{\varepsilon ^{7}}\sum_{\sigma _{1},\sigma _{2}=\pm 1}\sigma
_{1}\sigma _{2}\int d\mathbf{x}_{3}\int d\mathbf{v}_{3}%
\int_{0}^{t_{2}}dt_{3}\int_{\mathbb{R}^{3}}dh_{1}\int_{\mathbb{R}^{3}}dh_{2}
\\
&&e^{\frac{ih_{1}\cdot (x_{1}+(t_{2}-t_{3})\left( v_{1}+\sigma _{2}\frac{%
h_{2}}{2}\right) -x_{2}-(t_{2}-t_{3})v_{2})}{\varepsilon }}e^{\frac{%
ih_{2}\cdot \left( x_{1}-x_{3}\right) }{\varepsilon }}\hat{\phi}(h_{1})\hat{%
\phi}(h_{2}) \\
&&J(x_{1}+(t_{2}-t_{3})\left( v_{1}+\sigma _{2}\frac{h_{2}}{2}\right)
,v_{1}+\sigma _{2}\frac{h_{2}}{2}+\sigma _{1}\frac{h_{1}}{2}) \\
&&f_{N}^{(3)}(t_{3},x_{1},x_{2},x_{3},v_{1},v_{2},v_{3}).
\end{eqnarray*}%
Redo the change of variable: 
\begin{equation*}
t_{3}=t_{2}-\varepsilon s_{1},\text{ \ \ \ }h_{1}=\varepsilon \xi _{1}-h_{2},
\end{equation*}%
we then have%
\begin{eqnarray*}
&&\int J(x_{1},v_{1})Q_{1,3}^{\varepsilon ,1}f_{N}^{(3)}dx_{1}dv_{1} \\
&=&\frac{\varepsilon ^{4}}{\varepsilon ^{7}}\sum_{\sigma _{1},\sigma
_{2}=\pm 1}\sigma _{1}\sigma _{2}\int d\mathbf{x}_{3}\int d\mathbf{v}%
_{3}\int_{0}^{\frac{t_{2}}{\varepsilon }}ds_{1}\int_{\mathbb{R}^{3}}d\xi
_{1}\int_{\mathbb{R}^{3}}dh_{2} \\
&&e^{\frac{i\left( \varepsilon \xi _{1}-h_{2}\right) \cdot
(x_{1}+\varepsilon s_{1}\left( v_{1}+\sigma _{2}\frac{h_{2}}{2}\right)
-x_{2}-\varepsilon s_{1}v_{2})}{\varepsilon }}e^{\frac{ih_{2}\cdot \left(
x_{1}-x_{3}\right) }{\varepsilon }}\hat{\phi}(\varepsilon \xi _{1}-h_{2})%
\hat{\phi}(h_{2}) \\
&&J(x_{1}+\varepsilon s_{1}\left( v_{1}+\sigma _{2}\frac{h_{2}}{2}\right)
,v_{1}+\sigma _{2}\frac{h_{2}}{2}-\sigma _{1}\frac{h_{2}}{2}+\sigma _{1}%
\frac{\varepsilon \xi _{1}}{2}) \\
&&f_{N}^{(3)}(t_{2}-\varepsilon s_{1},x_{1},x_{2},x_{3},v_{1},v_{2},v_{3}).
\end{eqnarray*}%
Write out the phase,%
\begin{eqnarray*}
&=&\frac{1}{\varepsilon ^{3}}\sum_{\sigma _{1},\sigma _{2}=\pm 1}\sigma
_{1}\sigma _{2}\int d\mathbf{x}_{3}\int d\mathbf{v}_{3}\int_{0}^{\frac{t_{2}%
}{\varepsilon }}ds_{1}\int_{\mathbb{R}^{3}}d\xi _{1}\int_{\mathbb{R}%
^{3}}dh_{2} \\
&&e^{\frac{i\varepsilon \xi _{1}\cdot (x_{1}+\varepsilon s_{1}\left(
v_{1}+\sigma _{2}\frac{h_{2}}{2}\right) -x_{2}-\varepsilon s_{1}v_{2})}{%
\varepsilon }}e^{\frac{-ih_{2}\cdot (\varepsilon s_{1}\left( v_{1}+\sigma
_{2}\frac{h_{2}}{2}\right) -\varepsilon s_{1}v_{2})}{\varepsilon }}e^{\frac{%
-ih_{2}\cdot (x_{1}-x_{2})}{\varepsilon }}e^{\frac{ih_{2}\cdot \left(
x_{1}-x_{3}\right) }{\varepsilon }}\hat{\phi}(\varepsilon \xi _{1}-h_{2})%
\hat{\phi}(h_{2}) \\
&&J(x_{1}+\varepsilon s_{1}\left( v_{1}+\sigma _{2}\frac{h_{2}}{2}\right)
,v_{1}+\sigma _{2}\frac{h_{2}}{2}-\sigma _{1}\frac{h_{2}}{2}+\sigma _{1}%
\frac{\varepsilon \xi _{1}}{2}) \\
&&f_{N}^{(3)}(t_{2}-\varepsilon s_{1},x_{1},x_{2},x_{3},v_{1},v_{2},v_{3}).
\end{eqnarray*}%
Rearrange the phase,%
\begin{eqnarray*}
&=&\frac{1}{\varepsilon ^{3}}\sum_{\sigma _{1},\sigma _{2}=\pm 1}\sigma
_{1}\sigma _{2}\int d\mathbf{x}_{3}\int d\mathbf{v}_{3}\int_{0}^{\frac{t_{2}%
}{\varepsilon }}ds_{1}\int_{\mathbb{R}^{3}}d\xi _{1}\int_{\mathbb{R}%
^{3}}dh_{2} \\
&&e^{i\xi _{1}\cdot (x_{1}+\varepsilon s_{1}\left( v_{1}+\sigma _{2}\frac{%
h_{2}}{2}\right) -x_{2}-\varepsilon s_{1}v_{2})}e^{-ih_{2}\cdot (s_{1}\left(
v_{1}+\sigma _{2}\frac{h_{2}}{2}\right) -s_{1}v_{2})}e^{\frac{ih_{2}\cdot
\left( x_{2}-x_{3}\right) }{\varepsilon }}\hat{\phi}(\varepsilon \xi
_{1}-h_{2})\hat{\phi}(h_{2}) \\
&&J(x_{1}+\varepsilon s_{1}\left( v_{1}+\sigma _{2}\frac{h_{2}}{2}\right)
,v_{1}+\sigma _{2}\frac{h_{2}}{2}-\sigma _{1}\frac{h_{2}}{2}+\sigma _{1}%
\frac{\varepsilon \xi _{1}}{2}) \\
&&f_{N}^{(3)}(t_{2}-\varepsilon s_{1},x_{1},x_{2},x_{3},v_{1},v_{2},v_{3})
\end{eqnarray*}%
Now we need to perform one more change of variable to take care of $\left[
ih_{2}\cdot \left( x_{2}-x_{3}\right) \right] /\varepsilon $. We do%
\begin{equation}
x_{3,old}=x_{2}-\varepsilon x_{3,new},  \label{change of variable:Unknown}
\end{equation}%
and arrive at%
\begin{eqnarray*}
&&\int J(x_{1},v_{1})Q_{1,3}^{\varepsilon ,1}f_{N}^{(3)}dx_{1}dv_{1} \\
&=&\frac{\varepsilon ^{3}}{\varepsilon ^{3}}\sum_{\sigma _{1},\sigma
_{2}=\pm 1}\sigma _{1}\sigma _{2}\int d\mathbf{x}_{3}\int d\mathbf{v}%
_{3}\int_{0}^{\frac{t_{2}}{\varepsilon }}ds_{1}\int_{\mathbb{R}^{3}}d\xi
_{1}\int_{\mathbb{R}^{3}}dh_{2} \\
&&e^{i\xi _{1}\cdot (x_{1}+\varepsilon s_{1}\left( v_{1}+\sigma _{2}\frac{%
h_{2}}{2}\right) -x_{2}-\varepsilon s_{1}v_{2})}e^{-ih_{2}\cdot (s_{1}\left(
v_{1}+\sigma _{2}\frac{h_{2}}{2}\right) -s_{1}v_{2})}e^{ih_{2}\cdot x_{3}} \\
&&\hat{\phi}(\varepsilon \xi _{1}-h_{2})\hat{\phi}(h_{2})J(x_{1}+\varepsilon
s_{1}\left( v_{1}+\sigma _{2}\frac{h_{2}}{2}\right) ,v_{1}+\sigma _{2}\frac{%
h_{2}}{2}) \\
&&f_{N}^{(3)}(t_{2}-\varepsilon s_{1},x_{1},x_{2},x_{2}-\varepsilon
x_{3},v_{1},v_{2},v_{3})
\end{eqnarray*}%
which simplifies to%
\begin{eqnarray}
&=&\sum_{\sigma _{1},\sigma _{2}=\pm 1}\sigma _{1}\sigma _{2}\int d\mathbf{x}%
_{3}\int d\mathbf{v}_{3}\int_{0}^{\frac{t_{2}}{\varepsilon }}ds_{1}\int_{%
\mathbb{R}^{3}}d\xi _{1}\int_{\mathbb{R}^{3}}dh_{2}
\label{eqn:cubic term 1 before limit} \\
&&e^{i\xi _{1}\cdot (x_{1}+\varepsilon s_{1}\left( v_{1}+\sigma _{2}\frac{%
h_{2}}{2}\right) -x_{2}-\varepsilon s_{1}v_{2})}e^{-ih_{2}\cdot (s_{1}\left(
v_{1}+\sigma _{2}\frac{h_{2}}{2}\right) -s_{1}v_{2})}e^{ih_{2}\cdot x_{3}}%
\hat{\phi}(\varepsilon \xi _{1}-h_{2})\hat{\phi}(h_{2})  \notag \\
&&J(x_{1}+\varepsilon s_{1}\left( v_{1}+\sigma _{2}\frac{h_{2}}{2}\right)
,v_{1}+\sigma _{2}\frac{h_{2}}{2}-\sigma _{1}\frac{h_{2}}{2}+\sigma _{1}%
\frac{\varepsilon \xi _{1}}{2})  \notag \\
&&f_{N}^{(3)}(t_{2}-\varepsilon s_{1},x_{1},x_{2},x_{2}-\varepsilon
x_{3},v_{1},v_{2},v_{3}).  \notag
\end{eqnarray}%
Taking the $\varepsilon \rightarrow 0$ limit inside, which is justified in 
\S \ref{Sec:Interchange}, we have%
\begin{eqnarray}
&&\lim_{\varepsilon \rightarrow 0}\int J(x_{1},v_{1})Q_{1,3}^{\varepsilon
,1}f_{N}^{(3)}dx_{1}dv_{1}  \label{eqn:cubic term 1 after limit} \\
&=&\sum_{\sigma _{1},\sigma _{2}=\pm 1}\sigma _{1}\sigma _{2}\int d\mathbf{x}%
_{3}\int d\mathbf{v}_{3}\int_{0}^{\infty }ds_{1}\int_{\mathbb{R}^{3}}d\xi
_{1}\int_{\mathbb{R}^{3}}dh_{2}  \notag \\
&&e^{i\xi _{1}\cdot (x_{1}-x_{2})}e^{-ih_{2}\cdot (s_{1}\left( v_{1}+\sigma
_{2}\frac{h_{2}}{2}\right) -s_{1}v_{2})}e^{ih_{2}\cdot x_{3}}\hat{\phi}%
(-h_{2})\hat{\phi}(h_{2})  \notag \\
&&J(x_{1},v_{1}+\sigma _{2}\frac{h_{2}}{2}-\sigma _{1}\frac{h_{2}}{2}%
)f^{(3)}(t_{2},x_{1},x_{2},x_{2},v_{1},v_{2},v_{3}).  \notag
\end{eqnarray}%
Do the $dx_{3}$ (not $d\mathbf{x}_{3}$) integration, 
\begin{eqnarray}
&=&\sum_{\sigma _{1},\sigma _{2}=\pm 1}\sigma _{1}\sigma _{2}\int d\mathbf{x}%
_{2}\int d\mathbf{v}_{3}\int_{0}^{\infty }ds_{1}\int_{\mathbb{R}^{3}}d\xi
_{1}\int_{\mathbb{R}^{3}}dh_{2}  \notag \\
&&e^{i\xi _{1}\cdot (x_{1}-x_{2})}e^{-ih_{2}\cdot (s_{1}\left( v_{1}+\sigma
_{2}\frac{h_{2}}{2}\right) -s_{1}v_{2})}\delta (h_{2})\hat{\phi}(-h_{2})\hat{%
\phi}(h_{2})  \notag \\
&&J(x_{1},v_{1}+\sigma _{2}\frac{h_{2}}{2}-\sigma _{1}\frac{h_{2}}{2}%
)f^{(3)}(t_{2},x_{1},x_{2},x_{2},v_{1},v_{2},v_{3})  \notag
\end{eqnarray}%
Do the $dx_{2}d\xi _{1}dh_{2}$ integration,%
\begin{equation}
=\sum_{\sigma _{1},\sigma _{2}=\pm 1}\sigma _{1}\sigma _{2}\int dx_{1}\int d%
\mathbf{v}_{3}\int_{0}^{\infty }ds_{1}\left\vert \hat{\phi}(0)\right\vert
^{2}J(x_{1},v_{1})f^{(3)}(t_{2},x_{1},x_{1},x_{1},v_{1},v_{2},v_{3})  \notag
\end{equation}%
Since $\hat{\phi}(0)=0$, the above is zero and hence 
\begin{equation*}
\lim_{\varepsilon \rightarrow 0}\int J(x_{1},v_{1})Q_{1,3}^{\varepsilon
,1}f_{N}^{(3)}dx_{1}dv_{1}=0.
\end{equation*}%
Notice that if $\hat{\phi}(0)\neq 0$, the $ds_{1}$ integral yields an
infinity. We formally see that it is necessary for $\phi $ to have zero
integration in order to have the quantum Boltzmann hierarchy (\ref%
{hierarchy:QBHierarchy in differential form}) and hence the quantum
Boltzmann equation (\ref{eqn:proposed QBEquation}). We will go back to (\ref%
{eqn:cubic term 1 before limit}) in \S \ref{Sec:NecessaryCondition} to
discuss more about this. It is natural to wonder if $Q_{1,3}^{\varepsilon
,2}f_{N}^{(3)}$ will carry a negative sign and hence cancel out $%
Q_{1,3}^{\varepsilon ,1}f_{N}^{(3)}.$ Such a guess is not true. The term%
\begin{equation*}
\lim_{\varepsilon \rightarrow 0}\int J(x_{1},v_{1})Q_{1,3}^{\varepsilon
,2}f_{N}^{(3)}dx_{1}dv_{1}
\end{equation*}%
actually equals to (\ref{eqn:cubic term 1 after limit}) with no sign
difference. (See (\ref{eqn:cubic term 2 after limit}).) In below we treat $%
Q_{1,3}^{\varepsilon ,2}f_{N}^{(3)}.$

\subsubsection{Treatment of $Q_{1,3}^{\protect\varepsilon ,2}f_{N}^{(3)}$}

We write%
\begin{eqnarray*}
&&\int_{0}^{t_{2}}S^{(2)}(t_{2}-t_{3})B_{2,3}^{\varepsilon
}f_{N}^{(3)}(t_{3})dt_{3} \\
&=&\sum_{\sigma _{2}=\pm 1}\sigma
_{2}\int_{0}^{t_{2}}S^{(2)}(t_{2}-t_{3})dt_{3}\int_{\mathbb{R}%
^{3}}dx_{3}\int_{\mathbb{R}^{3}}dv_{3}\int_{\mathbb{R}^{3}}dh_{2}e^{\frac{%
ih_{2}\cdot (x_{2}-x_{3})}{\varepsilon }}\hat{\phi}(h_{2}) \\
&&f_{N}^{(3)}\left( t_{3},x_{1},x_{2},x_{3},v_{1},v_{2}-\sigma _{2}\frac{%
h_{2}}{2},v_{3}+\sigma _{2}\frac{h_{2}}{2}\right) \\
&=&\sum_{\sigma _{2}=\pm 1}\sigma _{2}\int_{0}^{t_{2}}dt_{3}\int_{\mathbb{R}%
^{3}}dx_{3}\int_{\mathbb{R}^{3}}dv_{3}\int_{\mathbb{R}^{3}}dh_{2}e^{\frac{%
ih_{2}\cdot \left[ \left( x_{2}-(t_{2}-t_{3})v_{2}\right) -x_{3}\right] }{%
\varepsilon }}\hat{\phi}(h_{2}) \\
&&f_{N}^{(3)}\left(
t_{3},x_{1}-(t_{2}-t_{3})v_{1},x_{2}-(t_{2}-t_{3})v_{2},x_{3},v_{1},v_{2}-%
\sigma _{2}\frac{h_{2}}{2},v_{3}+\sigma _{2}\frac{h_{2}}{2}\right)
\end{eqnarray*}%
then%
\begin{eqnarray*}
&&\frac{N^{2}}{\varepsilon }B_{1,2}^{\varepsilon
}\int_{0}^{t_{2}}S^{(2)}(t_{2}-t_{3})B_{2,3}^{\varepsilon
}f_{N}^{(3)}(t_{3})dt_{3} \\
&=&\frac{N^{2}}{\varepsilon }\sum_{\sigma _{1},\sigma _{2}=\pm 1}\sigma
_{1}\sigma _{2}\int_{0}^{t_{2}}dt_{3}\int_{\mathbb{R}^{3}}dx_{2}\int_{%
\mathbb{R}^{3}}dx_{3}\int_{\mathbb{R}^{3}}dv_{2}\int_{\mathbb{R}%
^{3}}dv_{3}\int_{\mathbb{R}^{3}}dh_{1}\int_{\mathbb{R}^{3}}dh_{2} \\
&&e^{\frac{ih_{1}\cdot (x_{1}-x_{2})}{\varepsilon }}e^{\frac{ih_{2}\cdot %
\left[ \left( x_{2}-(t_{2}-t_{3})\left( v_{2}+\sigma _{1}\frac{h_{1}}{2}%
\right) \right) -x_{3}\right] }{\varepsilon }}\hat{\phi}(h_{1})\hat{\phi}%
(h_{2}) \\
&&f_{N}^{(3)}(t_{3},x_{1}-(t_{2}-t_{3})\left( v_{1}-\sigma _{1}\frac{h_{1}}{2%
}\right) ,x_{2}-(t_{2}-t_{3})\left( v_{2}+\sigma _{1}\frac{h_{1}}{2}\right)
,x_{3}, \\
&&v_{1}-\sigma _{1}\frac{h_{1}}{2},v_{2}+\sigma _{1}\frac{h_{1}}{2}-\sigma
_{2}\frac{h_{2}}{2},v_{3}+\sigma _{2}\frac{h_{2}}{2})
\end{eqnarray*}%
thus%
\begin{eqnarray*}
&&\int J(x_{1},v_{1})Q_{1,3}^{\varepsilon ,2}f_{N}^{(3)}dx_{1}dv_{1} \\
&=&\frac{1}{\varepsilon ^{7}}\sum_{\sigma _{1},\sigma _{2}=\pm 1}\sigma
_{1}\sigma _{2}\int d\mathbf{x}_{3}\int d\mathbf{v}_{3}%
\int_{0}^{t_{2}}dt_{3}\int_{\mathbb{R}^{3}}dh_{1}\int_{\mathbb{R}^{3}}dh_{2}
\\
&&e^{\frac{ih_{1}\cdot (x_{1}-x_{2})}{\varepsilon }}e^{\frac{ih_{2}\cdot %
\left[ \left( x_{2}-(t_{2}-t_{3})\left( v_{2}+\sigma _{1}\frac{h_{1}}{2}%
\right) \right) -x_{3}\right] }{\varepsilon }}\hat{\phi}(h_{1})\hat{\phi}%
(h_{2})J(x_{1},v_{1}) \\
&&f_{N}^{(3)}(t_{3},x_{1}-(t_{2}-t_{3})\left( v_{1}-\sigma _{1}\frac{h_{1}}{2%
}\right) ,x_{2}-(t_{2}-t_{3})\left( v_{2}+\sigma _{1}\frac{h_{1}}{2}\right)
,x_{3}, \\
&&v_{1}-\sigma _{1}\frac{h_{1}}{2},v_{2}+\sigma _{1}\frac{h_{1}}{2}-\sigma
_{2}\frac{h_{2}}{2},v_{3}+\sigma _{2}\frac{h_{2}}{2}).
\end{eqnarray*}%
Again, we change variables to move all $h$'s away from $f_{N}^{(3)}$. We use
the new $x$-variables:%
\begin{eqnarray*}
x_{1,new} &=&x_{1,old}-(t_{2}-t_{3})\left( v_{1}-\sigma _{1}\frac{h_{1}}{2}%
\right) \\
x_{2,new} &=&x_{2,old}-(t_{2}-t_{3})\left( v_{2}+\sigma _{1}\frac{h_{1}}{2}%
\right)
\end{eqnarray*}%
which gives 
\begin{eqnarray*}
&&\int J(x_{1},v_{1})Q_{1,3}^{\varepsilon ,2}f_{N}^{(3)}dx_{1}dv_{1} \\
&=&\frac{1}{\varepsilon ^{7}}\sum_{\sigma _{1},\sigma _{2}=\pm 1}\sigma
_{1}\sigma _{2}\int d\mathbf{x}_{3}\int d\mathbf{v}_{3}%
\int_{0}^{t_{2}}dt_{3}\int_{\mathbb{R}^{3}}dh_{1}\int_{\mathbb{R}^{3}}dh_{2}
\\
&&e^{\frac{ih_{1}\cdot (x_{1}+(t_{2}-t_{3})\left( v_{1}-\sigma _{1}\frac{%
h_{1}}{2}\right) -x_{2}-(t_{2}-t_{3})\left( v_{2}+\sigma _{1}\frac{h_{1}}{2}%
\right) )}{\varepsilon }}e^{\frac{ih_{2}\cdot \left( x_{2}-x_{3}\right) }{%
\varepsilon }}\hat{\phi}(h_{1})\hat{\phi}(h_{2}) \\
&&J(x_{1}+(t_{2}-t_{3})\left( v_{1}-\sigma _{1}\frac{h_{1}}{2}\right) ,v_{1})
\\
&&f_{N}^{(3)}(t_{3},x_{1},x_{2},x_{3},v_{1}-\sigma _{1}\frac{h_{1}}{2}%
,v_{2}+\sigma _{1}\frac{h_{1}}{2}-\sigma _{2}\frac{h_{2}}{2},v_{3}+\sigma
_{2}\frac{h_{2}}{2}).
\end{eqnarray*}%
Then the new velocity variables:%
\begin{equation*}
v_{1,new}=v_{1,old}-\sigma _{1}\frac{h_{1}}{2},\text{ \ }%
v_{2,new}=v_{2,old}+\sigma _{1}\frac{h_{1}}{2}-\sigma _{2}\frac{h_{2}}{2},%
\text{ \ }v_{3,new}=v_{3,old}+\sigma _{2}\frac{h_{2}}{2}
\end{equation*}%
gives%
\begin{eqnarray*}
&&\int J(x_{1},v_{1})Q_{1,3}^{\varepsilon ,2}f_{N}^{(3)}dx_{1}dv_{1} \\
&=&\frac{1}{\varepsilon ^{7}}\sum_{\sigma _{1},\sigma _{2}=\pm 1}\sigma
_{1}\sigma _{2}\int d\mathbf{x}_{3}\int d\mathbf{v}_{3}%
\int_{0}^{t_{2}}dt_{3}\int_{\mathbb{R}^{3}}dh_{1}\int_{\mathbb{R}^{3}}dh_{2}
\\
&&e^{\frac{ih_{1}\cdot (x_{1}+(t_{2}-t_{3})v_{1}-x_{2}-(t_{2}-t_{3})\left(
v_{2}+\sigma _{2}\frac{h_{2}}{2}\right) )}{\varepsilon }}e^{\frac{%
ih_{2}\cdot \left( x_{2}-x_{3}\right) }{\varepsilon }}\hat{\phi}(h_{1})\hat{%
\phi}(h_{2}) \\
&&J(x_{1}+(t_{2}-t_{3})v_{1},v_{1}+\sigma _{1}\frac{h_{1}}{2}%
)f_{N}^{(3)}(t_{3},x_{1},x_{2},x_{3},v_{1},v_{2},v_{3}).
\end{eqnarray*}%
With the change of variables%
\begin{equation*}
t_{3}=t_{2}-\varepsilon s_{1},\text{ \ \ }h_{1}=\varepsilon \xi _{1}-h_{2},
\end{equation*}%
we then have%
\begin{eqnarray*}
&&\int J(x_{1},v_{1})Q_{1,3}^{\varepsilon ,2}f_{N}^{(3)}dx_{1}dv_{1} \\
&=&\frac{\varepsilon ^{4}}{\varepsilon ^{7}}\sum_{\sigma _{1},\sigma
_{2}=\pm 1}\sigma _{1}\sigma _{2}\int d\mathbf{x}_{3}\int d\mathbf{v}%
_{3}\int_{0}^{\frac{t_{2}}{\varepsilon }}ds_{1}\int_{\mathbb{R}^{3}}d\xi
_{1}\int_{\mathbb{R}^{3}}dh_{2} \\
&&e^{\frac{i\left( \varepsilon \xi _{1}-h_{2}\right) \cdot
(x_{1}+\varepsilon s_{1}v_{1}-x_{2}-\varepsilon s_{1}\left( v_{2}+\sigma _{2}%
\frac{h_{2}}{2}\right) )}{\varepsilon }}e^{\frac{ih_{2}\cdot \left(
x_{2}-x_{3}\right) }{\varepsilon }}\hat{\phi}(\varepsilon \xi _{1}-h_{2})%
\hat{\phi}(h_{2}) \\
&&J(x_{1}+\varepsilon s_{1}v_{1},v_{1}+\sigma _{1}\frac{\varepsilon \xi
_{1}-h_{2}}{2})f_{N}^{(3)}(t_{2}-\varepsilon
s_{1},x_{1},x_{2},x_{3},v_{1},v_{2},v_{3}).
\end{eqnarray*}%
Write out the phase,%
\begin{eqnarray*}
&=&\frac{1}{\varepsilon ^{3}}\sum_{\sigma _{1},\sigma _{2}=\pm 1}\sigma
_{1}\sigma _{2}\int d\mathbf{x}_{3}\int d\mathbf{v}_{3}\int_{0}^{\frac{t_{2}%
}{\varepsilon }}ds_{1}\int_{\mathbb{R}^{3}}d\xi _{1}\int_{\mathbb{R}%
^{3}}dh_{2} \\
&&e^{\frac{i\varepsilon \xi _{1}\cdot (x_{1}+\varepsilon
s_{1}v_{1}-x_{2}-\varepsilon s_{1}\left( v_{2}+\sigma _{2}\frac{h_{2}}{2}%
\right) )}{\varepsilon }}e^{\frac{-ih_{2}\cdot (\varepsilon
s_{1}v_{1}-\varepsilon s_{1}\left( v_{2}+\sigma _{2}\frac{h_{2}}{2}\right) )%
}{\varepsilon }}e^{\frac{-ih_{2}\cdot (x_{1}-x_{2})}{\varepsilon }}e^{\frac{%
ih_{2}\cdot \left( x_{2}-x_{3}\right) }{\varepsilon }} \\
&&\hat{\phi}(\varepsilon \xi _{1}-h_{2})\hat{\phi}(h_{2})J(x_{1}+\varepsilon
s_{1}v_{1},v_{1}+\sigma _{1}\frac{\varepsilon \xi _{1}-h_{2}}{2}) \\
&&f_{N}^{(3)}(t_{2}-\varepsilon s_{1},x_{1},x_{2},x_{3},v_{1},v_{2},v_{3}),
\end{eqnarray*}%
that is,%
\begin{eqnarray*}
&=&\frac{1}{\varepsilon ^{3}}\sum_{\sigma _{1},\sigma _{2}=\pm 1}\sigma
_{1}\sigma _{2}\int d\mathbf{x}_{3}\int d\mathbf{v}_{3}\int_{0}^{\frac{t_{2}%
}{\varepsilon }}ds_{1}\int_{\mathbb{R}^{3}}d\xi _{1}\int_{\mathbb{R}%
^{3}}dh_{2} \\
&&e^{i\xi _{1}\cdot (x_{1}+\varepsilon s_{1}v_{1}-x_{2}-\varepsilon
s_{1}\left( v_{2}+\sigma _{2}\frac{h_{2}}{2}\right) )}e^{-ih_{2}\cdot
(s_{1}v_{1}-s_{1}\left( v_{2}+\sigma _{2}\frac{h_{2}}{2}\right) )}e^{\frac{%
ih_{2}\cdot \left( 2x_{2}-x_{1}-x_{3}\right) }{\varepsilon }} \\
&&\hat{\phi}(\varepsilon \xi _{1}-h_{2})\hat{\phi}(h_{2})J(x_{1}+\varepsilon
s_{1}v_{1},v_{1}+\sigma _{1}\frac{\varepsilon \xi _{1}-h_{2}}{2}) \\
&&f_{N}^{(3)}(t_{2}-\varepsilon s_{1},x_{1},x_{2},x_{3},v_{1},v_{2},v_{3}).
\end{eqnarray*}%
Another change of variable%
\begin{equation*}
x_{3,old}=2x_{2}-x_{1}-\varepsilon x_{3,new},
\end{equation*}%
takes us to%
\begin{eqnarray}
&&\int J(x_{1},v_{1})Q_{1,3}^{\varepsilon ,2}f_{N}^{(3)}dx_{1}dv_{1}
\label{eqn:cubic term 2 before limit} \\
&=&\frac{\varepsilon ^{3}}{\varepsilon ^{3}}\sum_{\sigma _{1},\sigma
_{2}=\pm 1}\sigma _{1}\sigma _{2}\int d\mathbf{x}_{3}\int d\mathbf{v}%
_{3}\int_{0}^{\frac{t_{2}}{\varepsilon }}ds_{1}\int_{\mathbb{R}^{3}}d\xi
_{1}\int_{\mathbb{R}^{3}}dh_{2}  \notag \\
&&e^{i\xi _{1}\cdot (x_{1}+\varepsilon s_{1}v_{1}-x_{2}-\varepsilon
s_{1}\left( v_{2}+\sigma _{2}\frac{h_{2}}{2}\right) )}e^{-ih_{2}\cdot
(s_{1}v_{1}-s_{1}\left( v_{2}+\sigma _{2}\frac{h_{2}}{2}\right)
)}e^{ih_{2}\cdot x_{3}}  \notag \\
&&\hat{\phi}(\varepsilon \xi _{1}-h_{2})\hat{\phi}(h_{2})J(x_{1}+\varepsilon
s_{1}v_{1},v_{1}+\sigma _{1}\frac{\varepsilon \xi _{1}-h_{2}}{2})  \notag \\
&&f_{N}^{(3)}(t_{2}-\varepsilon s_{1},x_{1},x_{2},2x_{2}-x_{1}-\varepsilon
x_{3},v_{1},v_{2},v_{3}).  \notag
\end{eqnarray}%
Putting the $\varepsilon \rightarrow 0$ limit inside, which is justified in 
\S \ref{Sec:Interchange}, we have%
\begin{eqnarray}
&&\lim_{\varepsilon \rightarrow 0}\int J(x_{1},v_{1})Q_{1,3}^{\varepsilon
,2}f_{N}^{(3)}dx_{1}dv_{1}  \label{eqn:cubic term 2 after limit} \\
&=&\sum_{\sigma _{1},\sigma _{2}=\pm 1}\sigma _{1}\sigma _{2}\int d\mathbf{x}%
_{3}\int d\mathbf{v}_{3}\int_{0}^{\infty }ds_{1}\int_{\mathbb{R}^{3}}d\xi
_{1}\int_{\mathbb{R}^{3}}dh_{2}  \notag \\
&&e^{i\xi _{1}\cdot (x_{1}-x_{2})}e^{-ih_{2}\cdot (s_{1}v_{1}-s_{1}\left(
v_{2}+\sigma _{2}\frac{h_{2}}{2}\right) )}e^{ih_{2}\cdot x_{3}}\hat{\phi}%
(-h_{2})\hat{\phi}(h_{2})J(x_{1},v_{1}-\sigma _{1}\frac{h_{2}}{2})  \notag \\
&&f^{(3)}(t_{2},x_{1},x_{2},2x_{2}-x_{1},v_{1},v_{2},v_{3})  \notag \\
&=&\sum_{\sigma _{1},\sigma _{2}=\pm 1}\sigma _{1}\sigma _{2}\int d\mathbf{x}%
_{2}\int d\mathbf{v}_{3}\int_{\mathbb{R}^{3}}d\xi _{1}\int_{\mathbb{R}%
^{3}}dh_{2}  \notag \\
&&e^{i\xi _{1}\cdot (x_{1}-x_{2})}e^{-ih_{2}\cdot (s_{1}v_{1}-s_{1}\left(
v_{2}+\sigma _{2}\frac{h_{2}}{2}\right) )}\delta (h_{2})\hat{\phi}(-h_{2})%
\hat{\phi}(h_{2})  \notag \\
&&J(x_{1},v_{1}-\sigma _{1}\frac{h_{2}}{2}%
)f^{(3)}(t_{2},x_{1},x_{2},2x_{2}-x_{1},v_{1},v_{2},v_{3})  \notag
\end{eqnarray}%
which is zero under the same reasoning as in the treatment of $%
Q_{1,3}^{\varepsilon ,1}f_{N}^{(3)}.$

At this point, we have proven that the possible cubic term 
\begin{equation*}
V=\frac{N^{2}}{\varepsilon }\int_{0}^{t_{1}}S^{(1)}(t_{1}-t_{2})B_{%
\varepsilon }^{(2)}\int_{0}^{t_{2}}S^{(2)}(t_{2}-t_{3})B_{\varepsilon
}^{(3)}f_{N}^{(3)}(t_{3})dt_{3}dt_{2}
\end{equation*}%
is zero in the $\varepsilon \rightarrow 0$ limit. Therefore, we have proven
that relation (\ref{target:test with a test function}) holds for $f^{(k)}$
and hence established Theorem \ref{THM:MainTheorem}. The rest of this
section is to prove that we can take the limits inside the integrals under
the assumptions of Theorem \ref{THM:MainTheorem}.

\subsection{Justifying $\lim_{\protect\varepsilon \rightarrow 0}\protect\int %
=\protect\int \lim_{\protect\varepsilon \rightarrow 0}\label{Sec:Interchange}
$}

We interchanged "$\lim_{\varepsilon \rightarrow 0}"$ and "$\int $" in going
from going from (\ref{eqn:quadratic term before limit}) to (\ref%
{eqn:quardratic term after limit}), from (\ref{eqn:cubic term 1 before limit}%
) to (\ref{eqn:cubic term 1 after limit}), and from (\ref{eqn:cubic term 2
before limit}) to (\ref{eqn:cubic term 2 after limit}). We justify (\ref%
{eqn:cubic term 2 before limit}) to (\ref{eqn:cubic term 2 after limit}).
The proof of the other two is similar.

We claim that, if%
\begin{equation*}
\sup_{N}\sum_{j=1}^{3}\sum_{m=0}^{4}\left( \left\Vert \partial
_{x_{j}}^{m}f_{N}^{(3)}(t,\cdot )\right\Vert _{L^{1}}+\left\Vert \partial
_{v_{j}}^{m}f_{N}^{(3)}(t,\cdot )\right\Vert _{L^{1}}\right) <+\infty ,
\end{equation*}%
then let $\varepsilon \rightarrow 0$, we have%
\begin{eqnarray*}
&&\lim_{\varepsilon \rightarrow 0}\int J(x_{1},v_{1})Q_{1,3}^{\varepsilon
,2}f_{N}^{(3)}dx_{1}dv_{1} \\
&=&\sum_{\sigma _{1},\sigma _{2}=\pm 1}\sigma _{1}\sigma _{2}\int d\mathbf{x}%
_{3}\int d\mathbf{v}_{3}\int_{0}^{\infty }ds_{1}\int_{\mathbb{R}^{3}}d\xi
_{1}\int_{\mathbb{R}^{3}}dh_{2} \\
&&e^{i\xi _{1}\cdot (x_{1}-x_{2})}e^{-ih_{2}\cdot (s_{1}v_{1}-s_{1}\left(
v_{2}+\sigma _{2}\frac{h_{2}}{2}\right) )}e^{ih_{2}\cdot x_{3}}\hat{\phi}%
(-h_{2})\hat{\phi}(h_{2})J(x_{1},v_{1}-\sigma _{1}\frac{h_{2}}{2}) \\
&&f^{(3)}(t_{2},x_{1},x_{2},2x_{2}-x_{1},v_{1},v_{2},v_{3})
\end{eqnarray*}%
if $\int J(x_{1},v_{1})Q_{1,3}^{\varepsilon ,2}f_{N}^{(3)}dx_{1}dv_{1}$ is
given by (\ref{eqn:cubic term 1 before limit}).

In fact, rewrite%
\begin{equation*}
\int J(x_{1},v_{1})Q_{1,3}^{\varepsilon ,2}f_{N}^{(3)}dx_{1}dv_{1}=\int_{0}^{%
\frac{t_{2}}{\varepsilon }}ds_{1}\int_{\mathbb{R}^{3}}d\xi
_{1}A_{\varepsilon }(s_{1},\xi _{1})
\end{equation*}%
where 
\begin{eqnarray*}
A_{\varepsilon }(s_{1},\xi _{1}) &\equiv &\int d\mathbf{x}_{3}\int d\mathbf{v%
}_{3}\int_{\mathbb{R}^{3}}dh_{2} \\
&&e^{i\xi _{1}\cdot (x_{1}+\varepsilon s_{1}v_{1}-x_{2}-\varepsilon
s_{1}\left( v_{2}+\sigma _{2}\frac{h_{2}}{2}\right) )}e^{-ih_{2}\cdot
s_{1}[v_{1}-\left( v_{2}+\sigma _{2}\frac{h_{2}}{2}\right) ]}e^{ih_{2}\cdot
x_{3}} \\
&&\hat{\phi}(\varepsilon \xi _{1}-h_{2})\hat{\phi}(h_{2})J(x_{1}+\varepsilon
s_{1}v_{1},v_{1}+\sigma _{1}\frac{\varepsilon \xi _{1}-h_{2}}{2}) \\
&&f_{N}^{(3)}(t_{2}-\varepsilon s_{1},x_{1},x_{2},2x_{2}-x_{1}-\varepsilon
x_{3},v_{1},v_{2},v_{3}).
\end{eqnarray*}%
Let $x_{1}-x_{2}=y_{1},x_{1}+x_{2}=y_{2}$ as well as 
\begin{equation*}
v_{1}-v_{2}-\sigma _{2}\frac{h_{2}}{2}=w_{1},\text{ \ \ }v_{1}+v_{2}-\sigma
_{2}\frac{h_{2}}{2}=w_{2},\text{ \ \ }h_{2}=h_{2},
\end{equation*}%
which makes 
\begin{equation*}
v_{1}=\frac{w_{1}+w_{2}-2\sigma _{2}h_{2}}{2},\text{ \ \ \ }v_{2}=\frac{%
w_{1}-w_{2}}{2},\text{ \ \ \ }h_{2}=h_{2},
\end{equation*}%
we can then transform $A_{\varepsilon }(s_{1},\xi _{1})$ into 
\begin{eqnarray*}
&&A_{\varepsilon }(s_{1},\xi _{1}) \\
&=&\int e^{i\xi _{1}\cdot \{y_{1}+\varepsilon s_{1}[\frac{%
w_{1}+w_{2}-2\sigma _{2}h_{2}}{2}]-\varepsilon s_{1}\left( \frac{w_{1}-w_{2}%
}{2}+\sigma _{2}\frac{h_{2}}{2}\right) \}}e^{-ih_{2}\cdot
s_{1}w_{1}}e^{ih_{2}\cdot x_{3}}\hat{\phi}(\varepsilon \xi _{1}-h_{2})\hat{%
\phi}(h_{2}) \\
&&J(\frac{y_{1}+y_{2}}{2}+\varepsilon s_{1}\frac{w_{1}+w_{2}-2\sigma
_{2}h_{2}}{2},\frac{w_{1}+w_{2}-2\sigma _{2}h_{2}}{2}+\sigma _{1}\frac{%
\varepsilon \xi _{1}-h_{2}}{2}) \\
&&f_{N}^{(3)}(t_{2}-\varepsilon s_{1},\frac{y_{1}+y_{2}}{2},\frac{y_{1}-y_{2}%
}{2},\frac{y_{1}}{2}-\frac{3y_{2}}{2}-\varepsilon x_{3},\frac{%
w_{1}+w_{2}-2\sigma _{2}h_{2}}{2},\frac{w_{1}-w_{2}}{2},v_{3}). \\
&\equiv &\int e^{i\xi _{1}\cdot y_{1}}e^{-ih_{2}\cdot
s_{1}w_{1}}e^{ih_{2}\cdot x_{3}}B,
\end{eqnarray*}%
where $B(\varepsilon s_{1},y_{1},y_{2},\varepsilon
x_{3},w_{1},w_{2},\varepsilon \xi _{1},h_{2},v_{3})\equiv $%
\begin{eqnarray}
&&e^{i\xi _{1}\cdot \varepsilon s_{1}\{[\frac{w_{1}+w_{2}-2\sigma _{2}h_{2}}{%
2}]-\left( \frac{w_{1}-w_{2}}{2}+\sigma _{2}\frac{h_{2}}{2}\right) \}}\hat{%
\phi}(\varepsilon \xi _{1}-h_{2})\hat{\phi}(h_{2})  \notag \\
&&J(\frac{y_{1}+y_{2}}{2}+\varepsilon s_{1}\frac{w_{1}+w_{2}-2\sigma
_{2}h_{2}}{2},\frac{w_{1}+w_{2}-2\sigma _{2}h_{2}}{2}+\sigma _{1}\frac{%
\varepsilon \xi _{1}-h_{2}}{2})  \label{b} \\
&&f_{N}^{(3)}(t_{2}-\varepsilon s_{1},\frac{y_{1}+y_{2}}{2},\frac{y_{1}-y_{2}%
}{2},\frac{y_{1}}{2}-\frac{3y_{2}}{2}-\varepsilon x_{3},\frac{%
w_{1}+w_{2}-2\sigma _{2}h_{2}}{2},\frac{w_{1}-w_{2}}{2},v_{3})  \notag
\end{eqnarray}%
Clearly, for bounded $\xi _{1},x_{3}$ and $s_{1},$such a integral is finite.
We only need to control large $\xi _{1},x_{3}$ and $s_{1}$ to pass to the
limit.

Upon using standard\ smooth cutoff functions, we only need to concentrate
the most singular region of 
\begin{equation}
\left\vert \xi _{1}\right\vert \geqslant 1,\left\vert x_{3}\right\vert
\geqslant 1,\left\vert s_{1}\right\vert \geqslant 1.  \label{region}
\end{equation}
we may further assume that in such a region, 
\begin{equation}
\left\vert \xi _{1}^{1}\right\vert \gtrsim |\xi _{1}|,\text{ \ \ }\left\vert
x_{3}^{1}\right\vert \gtrsim |x_{3}|,\text{ \ }|h_{2}^{1}|\gtrsim |h_{2}|
\label{component}
\end{equation}%
All the other cases are simpler and can be controlled similarly.

We first integrate by part in $y_{1}^{1},w_{1}^{1}$ repeatedly, (since $%
\varepsilon s_{1}$ is bounded), to obtain%
\begin{eqnarray*}
&&\int dy_{1}dy_{2}dx_{3}dw_{1}dw_{2}dv_{3}dh_{2}\frac{1}{\{\xi
_{1}^{1}\}^{m}s_{1}^{m}\{h_{2}^{1}\}^{m}}e^{i\xi _{1}\cdot
y_{1}}e^{-ih_{2}\cdot s_{1}w_{1}}e^{ih_{2}\cdot x_{3}} \\
&&\partial _{y_{1}^{1}}^{m}\partial _{w_{1}^{1}}^{m}B(\varepsilon
s_{1},y_{1},y_{2},x_{3},w_{1},w_{2},\xi _{1},h_{2},v_{3}) \\
&=&\int dy_{1}dy_{2}dx_{3}dw_{1}dw_{2}dv_{3}dh_{2}\frac{1}{\{\xi
_{1}^{1}\}^{m}s_{1}^{m}\{h_{2}^{1}\}^{m}}e^{i\xi _{1}\cdot
y_{1}}e^{-ih_{2}\cdot s_{1}w_{1}} \\
&&\frac{1}{x_{3}^{1}}\frac{d\{e^{ih_{2}^{1}x_{3}^{1}}\}}{dh_{2}^{1}}\partial
_{y_{1}^{1}}^{m}\partial _{w_{1}^{1}}^{m}B(\varepsilon
s_{1},y_{1},y_{2},x_{3},w_{1},w_{2},\xi _{1},h_{2},v_{3}).
\end{eqnarray*}%
we then take integration by part in $h_{2}^{1}$ \ four times\ as above to
get the worse term, in terms of vanishing order of $\hat{\phi}(\varepsilon
\xi _{1}-h_{2})\hat{\phi}(h_{2})$ in $B$ as 
\begin{eqnarray*}
&\backsim &\sum_{j}\int dy_{1}dy_{2}dx_{3}dw_{1}dw_{2}dv_{3}dh_{2}\frac{%
s_{1}^{j}\{w_{1}^{1}\}^{j}}{\{\xi
_{1}^{1}\}^{m}s_{1}^{m}\{h_{2}^{1}\}^{m}\{x_{1}^{1}\}^{j}}e^{i\xi _{1}\cdot
y_{1}}e^{-ih_{2}\cdot s_{1}w_{1}}e^{ih_{2}\cdot x_{3}} \\
&&\partial _{h_{2}^{1}}^{4-j}\partial _{h_{2}^{1}}^{m}\partial
_{y_{1}^{1}}^{m}\partial _{w_{1}^{1}}^{m}B(\varepsilon
s_{1},y_{1},y_{2},\varepsilon x_{3},w_{1},w_{2},\varepsilon \xi
_{1},h_{2},v_{3}) \\
&\backsim &\sum_{j=0}^{m}\int dy_{1}dy_{2}dx_{3}dw_{1}dw_{2}dv_{3}dh_{2}%
\frac{s_{1}^{j}\{w_{1}^{1}\}^{j}}{\{\xi
_{1}^{1}\}^{m}s_{1}^{m}\{h_{2}^{1}\}^{m}\{x_{1}^{1}\}^{j}}e^{i\xi _{1}\cdot
y_{1}}e^{-ih_{2}\cdot s_{1}w_{1}}e^{ih_{2}\cdot x_{3}} \\
&&\partial _{h_{2}^{1}}^{4-j+m}[\hat{\phi}(h_{2})\hat{\phi}(\varepsilon \xi
_{1}-h_{2})]\times B_{j}(\varepsilon s_{1},y_{1},y_{2},\varepsilon
x_{3},w_{1},w_{2},\varepsilon \xi _{1},h_{2},v_{3}).
\end{eqnarray*}%
Here $B_{j}$ is some nice function with decay in $w_{2}$ so that the growth
in $w_{1}^{1}$ is under control. Hence, this is uniformly integrable\ for
large $\xi _{1},s_{1},$ $x_{3},h_{2}$ if $m>5$ by (\ref{component}) and (\ref%
{region}). It suffices to control small $h_{2}^{1}$ for $|h_{2}^{1}|<1.$

We now use the vanishing condition of $\hat{\phi}$: $\hat{\phi}%
(h_{2})=h_{2}^{n}$ for $|h_{2}|\leq 1,$ so that for $0\leq j\leq 4,$%
\begin{equation*}
|\partial _{h_{2}^{1}}^{m+4-j}[\hat{\phi}(h_{2})\hat{\phi}(\varepsilon \xi
_{1}-h_{2})]|\leq h_{2}^{n-m-4}.
\end{equation*}%
Hence by (\ref{component}) near $h_{2}=0\,$, the integral has a singularity
of $\frac{1}{h_{2}^{4-n+2m}}.$ If $4-n+2m<3,$ or $n\geq 1+2m>11,$ then we
know that $\int_{|h_{2}|\leq 1}\frac{1}{h_{2}^{4-n+m}}dh_{2}<+\infty ,$ and
it is uniformly bounded integrable, by (\ref{component}). Hence, we can
interchange "$\lim_{\varepsilon \rightarrow 0}$" and "$\int $" in going from
(\ref{eqn:cubic term 2 before limit}) to (\ref{eqn:cubic term 2 after limit}%
) as claimed.

\subsection{It is Necessary to Have $\protect\int \protect\phi =0\label%
{Sec:NecessaryCondition}$}

Recall (\ref{eqn:cubic term 1 before limit}),%
\begin{eqnarray}
&&\int J(x_{1},v_{1})Q_{1,3}^{\varepsilon ,1}f_{N}^{(3)}dx_{1}dv_{1}
\label{eqn:cubic term 1 before limit for necessasity} \\
&=&\sum_{\sigma _{1},\sigma _{2}=\pm 1}\sigma _{1}\sigma _{2}\int d\mathbf{x}%
_{3}\int d\mathbf{v}_{3}\int_{0}^{\frac{t_{2}}{\varepsilon }}ds_{1}\int_{%
\mathbb{R}^{3}}d\xi _{1}\int_{\mathbb{R}^{3}}dh_{2}  \notag \\
&&e^{i\xi _{1}\cdot (x_{1}+\varepsilon s_{1}\left( v_{1}+\sigma _{2}\frac{%
h_{2}}{2}\right) -x_{2}-\varepsilon s_{1}v_{2})}e^{-ih_{2}\cdot (s_{1}\left(
v_{1}+\sigma _{2}\frac{h_{2}}{2}\right) -s_{1}v_{2})}e^{ih_{2}\cdot x_{3}}%
\hat{\phi}(\varepsilon \xi _{1}-h_{2})\hat{\phi}(h_{2})  \notag \\
&&J(x_{1}+\varepsilon s_{1}\left( v_{1}+\sigma _{2}\frac{h_{2}}{2}\right)
,v_{1}+\sigma _{2}\frac{h_{2}}{2}-\sigma _{1}\frac{h_{2}}{2}+\sigma _{1}%
\frac{\varepsilon \xi _{1}}{2})  \notag \\
&&f_{N}^{(3)}(t_{2}-\varepsilon s_{1},x_{1},x_{2},x_{2}-\varepsilon
x_{3},v_{1},v_{2},v_{3}).  \notag
\end{eqnarray}%
To see that it is necessary to have $\int \phi =0$ in order to have a $%
\varepsilon \rightarrow 0$ limit for the BBGKY hierarchy (\ref%
{hierarchy:QB-BBGKY in Differential Form}) and hence a possible derivation
for the quantum Boltzmann hierarchy (\ref{hierarchy:QBHierarchy in
differential form}) and the quantum Boltzmann equation (\ref{eqn:proposed
QBEquation}), we analysis the size of (\ref{eqn:cubic term 1 before limit
for necessasity}). To avoid some technical issues, let us assume that $%
J(...)f_{N}^{(3)}(...)$ is a test function $g,$ because the main point here
is the phase 
\begin{equation*}
e^{-ih_{2}\cdot (s_{1}\left( v_{1}+\sigma _{2}\frac{h_{2}}{2}\right)
-s_{1}v_{2})}e^{ih_{2}\cdot x_{3}}.
\end{equation*}%
Do the $dx_{3}$ integrals in (\ref{eqn:cubic term 1 before limit for
necessasity}), we have%
\begin{eqnarray*}
&=&\sum_{\sigma _{1},\sigma _{2}=\pm 1}\sigma _{1}\sigma _{2}\int d\mathbf{x}%
_{2}\int d\mathbf{v}_{3}\int_{0}^{\frac{t_{2}}{\varepsilon }}ds_{1}\int_{%
\mathbb{R}^{3}}d\xi _{1}\int_{\mathbb{R}^{3}}dh_{2} \\
&&e^{i\xi _{1}\cdot (x_{1}+\varepsilon s_{1}\left( v_{1}+\sigma _{2}\frac{%
h_{2}}{2}\right) -x_{2}-\varepsilon s_{1}v_{2})}e^{-ih_{2}\cdot (s_{1}\left(
v_{1}+\sigma _{2}\frac{h_{2}}{2}\right) -s_{1}v_{2})} \\
&&\hat{\phi}(\varepsilon \xi _{1}-h_{2})\hat{\phi}(h_{2})\frac{1}{%
\varepsilon ^{3}}\hat{g}(t_{2}-\varepsilon s_{1},x_{1},x_{2},\frac{h_{2}}{%
\varepsilon },v_{1},v_{2},v_{3}).
\end{eqnarray*}%
Here $\hat{g}$ means the Fourier transform in $x_{3}$. We then find that,
for every $t_{2}$, $\mathbf{x}_{2}$, $\mathbf{v}_{3}$, we effectively have a 
$\delta (h_{2})$, so that the $dh_{2}$ integral is restricted to have size $%
\left\vert h_{2}\right\vert \lesssim \varepsilon .$ Now, say $t_{2}\lesssim
1 $, we know $\left\vert s_{1}h_{2}\right\vert \lesssim 1$ and hence%
\begin{equation*}
e^{-ih_{2}\cdot (s_{1}\left( v_{1}+\sigma _{2}\frac{h_{2}}{2}\right)
-s_{1}v_{2})}\sim 1
\end{equation*}%
which then makes the $ds_{1}$ integral to blow up as $\varepsilon
\rightarrow 0$ if $\hat{\phi}(0)\neq 0.$

\appendix

\section{The Cubic Term of (\protect\ref{eqn:U-U}) when $\hat{\protect\phi}%
(0)=0$}

Theorem \ref{THM:MainTheorem} is unexpected. It rules out the possibility to
have the Uehling-Uhlenbeck equation (\ref{eqn:U-U}) as the mean-field
equation from the BBGKY hierarchy (\ref{hierarchy:QB-BBGKY in Differential
Form}) in the sense that if $f$ solves (\ref{eqn:U-U}), then 
\begin{equation*}
f^{(k)}(t,\mathbf{x}_{k},\mathbf{v}_{k})=\dprod%
\limits_{j=1}^{k}f(t,x_{j},v_{j})
\end{equation*}%
is not a solution to the $N\rightarrow \infty $ limit of the BBGKY hierarchy
(\ref{hierarchy:QB-BBGKY in Differential Form}). It is then natural to
wonder if the assumption: $\hat{\phi}(0)=0,$ implies that the cubic term in
the Uehling-Uhlenbeck equation (\ref{eqn:U-U}) is zero. Such an statement is
unlikely to be true. We include a discussion here for completeness. On the
one hand, if $\hat{\phi}(0)\neq 0$, the $\varepsilon \rightarrow 0$ limit
for the BBGKY hierarchy (\ref{hierarchy:QB-BBGKY in Differential Form}) has
an infinite cubic term as shown formally in the proof of Theorem \ref%
{THM:MainTheorem} and in \S \ref{Sec:NecessaryCondition}. On the other hand,
recall the cubic term in (\ref{eqn:U-U})%
\begin{equation*}
M\left( f\right) =8\pi ^{3}\theta \int dv_{\ast }\int dv_{\ast }^{\prime
}\int dv^{\prime }W\left( v,v^{\prime }|v_{\ast },v_{\ast }^{\prime }\right) 
\left[ \left( f^{\prime }f_{\ast }^{\prime }f+f^{\prime }f_{\ast }^{\prime
}f_{\ast }\right) -\left( ff_{\ast }f_{\ast }^{\prime }+ff_{\ast }f^{\prime
}\right) \right] ,
\end{equation*}%
and%
\begin{eqnarray*}
W\left( v,v^{\prime }|v_{\ast },v_{\ast }^{\prime }\right) &=&\frac{1}{8\pi
^{2}}\left[ \hat{\phi}(v^{\prime }-v)+\theta \hat{\phi}(v^{\prime }-v_{\ast
})\right] ^{2}\delta (v+v_{\ast }-v^{\prime }-v_{\ast }^{\prime }) \\
&&\times \delta (\frac{1}{2}\left( v^{2}+v_{\ast }^{2}-\left( v^{\prime
}\right) ^{2}-\left( v_{\ast }^{\prime }\right) ^{2}\right) ).
\end{eqnarray*}%
Let us suppress the $(t,x)$ dependence in $f,$ $f^{\prime },f_{\ast },$ $%
f_{\ast }^{\prime }$ and write 
\begin{equation*}
f=f(v),\text{ }f^{\prime }=f(v^{\prime }),f_{\ast }=f(v_{\ast }),\text{ }%
f_{\ast }^{\prime }=f(v_{\ast }^{\prime }).
\end{equation*}%
since the integral we are considering has nothing to do with that. With the
usual parametrization:%
\begin{equation*}
v^{\prime }=v+\left[ \left( v-v_{\ast }\right) \cdot \omega \right] \omega ,%
\text{ \ \ \ }v_{\ast }^{\prime }=v_{\ast }-\left[ \left( v-v_{\ast }\right)
\cdot \omega \right] \omega ,
\end{equation*}%
we reach%
\begin{eqnarray*}
M\left( f\right) &=&\pi \theta \int_{\mathbb{R}^{3}}dv_{\ast }\int_{\mathbb{S%
}^{2}}dS_{\omega }\left\vert v-v_{\ast }\right\vert \\
&&\times \left[ \hat{\phi}(\left[ \left( v-v_{\ast }\right) \cdot \omega %
\right] \omega )+\theta \hat{\phi}\left( \left( v-v_{\ast }\right) +\left[
\left( v-v_{\ast }\right) \cdot \omega \right] \omega \right) \right] ^{2} \\
&&\times \left[ f^{\prime }f_{\ast }^{\prime }\left( f+f_{\ast }\right)
-ff_{\ast }\left( f^{\prime }+f_{\ast }^{\prime }\right) \right] .
\end{eqnarray*}%
Let $\theta =1.$ Assume that $\hat{\phi}$ does not change sign and $\hat{\phi%
}\left( \xi \right) =0$ only at $\xi =0$\footnote{%
Such a potential $\phi $ could be constructed.}, then%
\begin{equation*}
\left\vert v-v_{\ast }\right\vert \left[ \hat{\phi}(\left[ \left( v-v_{\ast
}\right) \cdot \omega \right] \omega )+\theta \hat{\phi}\left( \left(
v-v_{\ast }\right) +\left[ \left( v-v_{\ast }\right) \cdot \omega \right]
\omega \right) \right] ^{2}=0
\end{equation*}%
only when $v_{\ast }=v$ which is a measure zero set. It is then hard to
believe that if $\hat{\phi}\left( \xi \right) =0$ only at $\xi =0$ will make 
$M(f)=0$ for every $f$.

\end{document}